\documentclass[onefignum,onetabnum]{siamart190516}

\usepackage{graphicx}
\usepackage{multirow,bigstrut}
\usepackage{amsmath,amsfonts,amssymb}
\usepackage{mathrsfs}
\usepackage{color}
\usepackage{upgreek}
\usepackage{bm}
\usepackage{verbatim}
\usepackage{mathtools}
\usepackage{calligra}
\usepackage[T1]{fontenc}
\usepackage{subfigure}
\usepackage{pgfplots}
\usepackage{xcolor}

\usepackage{tikz}
  \usetikzlibrary{calc,shapes,decorations,decorations.text, mindmap,shadings,patterns,matrix,arrows,intersections,automata,backgrounds}
  
\usetikzlibrary{shapes,arrows}

\usepackage{verbatim}
\usepackage{exscale}
\usepackage{relsize}
\theoremstyle{plain}
\newtheorem{rem}[theorem]{\hspace{1mm}Remark}

\newtheorem{assumption}[theorem]{Assumption}

\makeatletter
\tikzset{
        hatch distance/.store in=\hatchdistance,
        hatch distance=5pt,
        hatch thickness/.store in=\hatchthickness,
        hatch thickness=5pt
        }
\pgfdeclarepatternformonly[\hatchdistance,\hatchthickness]{north east hatch}
    {\pgfqpoint{-1pt}{-1pt}}
    {\pgfqpoint{\hatchdistance}{\hatchdistance}}
    {\pgfpoint{\hatchdistance-1pt}{\hatchdistance-1pt}}%
    {
        \pgfsetcolor{\tikz@pattern@color}
        \pgfsetlinewidth{\hatchthickness}
        \pgfpathmoveto{\pgfqpoint{0pt}{0pt}}
        \pgfpathlineto{\pgfqpoint{\hatchdistance}{\hatchdistance}}
        \pgfusepath{stroke}
    }
\makeatother

\newcommand{\mm}[1]{{\color{black}{#1}}}

\newcommand{\rmm}[1]{{\color{black}{#1}}}
 
\begin{document}
 

\headers{GFEM for heterogeneous reaction-diffusion equations}{C. P. Ma, and M. J. Melenk} 
 
\title{Exponential convergence of a generalized FEM for heterogeneous reaction-diffusion equations}

\author{Chupeng Ma\thanks 
 {Institute of Scientific Research, Great Bay University, Songshan Lake International Innovation Entrepreneurship A5, Dongguan 523000, China (\email{chupeng.ma@gbu.edu.cn}).}
\and J. M. Melenk\thanks{Institut f\"{u}r Analysis und Scientific Computing, Vienna University of Technology, Wiedner Hauptstrasse 8-10, 1040 Wien, Austria (\email{melenk@tuwien.ac.at}).}}

\maketitle
\begin{abstract}
A generalized finite element method is proposed for solving a heterogeneous reaction-diffusion equation with a singular perturbation parameter $\varepsilon$, based on locally approximating the solution on each subdomain by solution of a local reaction-diffusion equation and eigenfunctions of a local eigenproblem. These local problems are posed on some domains slightly larger than the subdomains with oversampling size $\delta^{\ast}$. The method is formulated at the continuous level as a direct discretization of the continuous problem and at the discrete level as a coarse-space approximation for its standard FE discretizations. Exponential decay rates for local approximation errors with respect to $\delta^{\ast}/\varepsilon$ and $\delta^{\ast}/h$ (at the discrete level with $h$ denoting the fine FE mesh size) and with the local degrees of freedom are established. In particular, it is shown that the method at the continuous level converges uniformly with respect to $\varepsilon$ in the standard $H^{1}$ norm, and that if the oversampling size is relatively large with respect to $\varepsilon$ and $h$ (at the discrete level), the solutions of the local reaction-diffusion equations provide good local approximations for the solution and thus the local eigenfunctions are not needed. Numerical results are provided to verify the theoretical results. 

\end{abstract}
\begin{keywords}
generalized finite element method, multiscale method, reaction-diffusion equation, singular perturbation, local approximation
\end{keywords}
\begin{AMS}
65M60, 65N15, 65N55
\end{AMS}

\section{Introduction}
Singularly perturbed reaction-diffusion equations has been extensively studied in the numerical analysis community; see, e.g., \cite{kadalbajoo2010brief,roos2008robust} and the references therein. It is well known that the solution of these equations typically exhibits sharp boundary layers, making the numerical approximation notoriously difficult. Significant research efforts have focused on devising parameter-robust numerical methods, i.e., methods with a uniform convergence rate with respect to the singular perturbation parameter. Such methods have been commonly designed by using layer-adapted meshes, such as the Bakhvalov mesh \cite{bakhvalov1969optimization}, the Shishkin mesh \cite{shishkin1989grid}, and the Spectral Boundary Layer mesh in the context of the p/hp-version FEM \cite{melenk1997robust,melenk2002hp}, and uniform convergence rates were typically obtained for problems with sufficiently smooth data. In recent years, parameter robust methods with error estimates in a so-called balanced norm have attracted considerable research interest \cite{lin2012balanced,melenk2016robust,roos2015convergence}. 

However, although singularly perturbed reaction-diffusion equations have already been intensively studied for decades and numerous numerical methods have been developed, the existing works, almost without exception, only dealt with problems with constant or of smooth diffusion coefficients. This is mainly due to the fact that the solution is typically required to be sufficiently smooth in (parameter-uniform) error estimates of these methods. The singularly perturbed heterogeneous reaction-diffusion problems, possibly with a multiscale character in the diffusion coefficient, have received little attention. Such problems arise naturally in several applications, e.g., in implicit time-discretizations with small time steps of wave or parabolic problems in heterogeneous media, and in the numerical solution of a so-called spectral fractional diffusion equation with a heterogeneous diffusion coefficient \cite{banjai2020exponential}. Compared with the constant coefficient case, the numerical approximation and analysis of singularly perturbed problems with highly varying coefficients are much more challenging, due to the following two reasons. First, in addition to resolving the layers with strongly refined meshes near the boundary of the domain, it also requires a very fine mesh in the interior to resolve the strongly heterogeneous diffusion coefficient, making a direct discretization with an acceptable accuracy prohibitively expensive computationally. Second, the regularity of the solution of such problems is typically too low to derive (uniform) error estimates using existing numerical analysis techniques. 

Numerical multiscale methods, mostly developed in the framework of Galerkin methods, have proved effective in addressing the difficulty associated with rough coefficients in PDEs, particularly elliptic type PDEs. In contrast to standard FEMs using classical polynomial basis functions, numerical multiscale methods use problem-dependent, precomputed local basis functions that encode the fine-scale information of the solution, leading to a much higher computational efficiency as significantly fewer basis functions are needed for comparable accuracy. For heterogeneous convection-diffusion problems in the convection-dominated regime, various multiscale methods have been developed, e.g., variational multiscale methods \cite{maalqvist2011multiscale}, (generalized) multiscale finite element methods \cite{le2017numerical,park2004multiscale,zhao2022constraint}, multiscale discontinuous Galerkin methods \cite{kim2014multiscale}, heterogeneous multiscale methods \cite{abdulle2014discontinuous}, multiscale hybrid-mixed methods \cite{harder2015multiscale}, and multiscale stabilization methods \cite{chung2020multiscale,calo2016multiscale}. Most of the existing works on this subject focused on algorithm development and rigorous theoretical analysis is much less common. In contrast, there is little work on multiscale methods for reaction-dominated reaction-diffusion problems with heterogeneous coefficients, although the idea of numerical multiscale methods has been employed to construct stabilized and enriched FEMs for standard singularly perturbed reaction-diffusion problems \cite{franca2005convergence,franca2005towards,fernando2012numerical}. This paper aims to fill the gap. 

We are concerned with a particular multiscale method, the Multiscale Spectral Generalized Finite Element Method (MS-GFEM). The method was first introduced in the pioneering work \cite{babuska2011optimal} for solving heterogeneous diffusion problems and was later extended to heterogeneous elasticity problems \cite{babuvska2014machine} and parabolic problems \cite{schleuss2022optimal}. Developed in the framework of the partition of unity method \cite{babuvska1997partition}, the MS-GFEM builds optimal local approximation spaces from eigenfunctions of specially designed local eigenproblems. It was rigorously proved that the local spectral basis, augmented with the solution of a local problem involving the same differential operator as the original equation, can approximate the exact solution locally with error decaying exponentially with respect to the local degrees of freedom. A crucial ingredient of the method that triggers the exponential error decay is the \emph{oversampling} technique, i.e., the local problems are solved on some domains slighter larger than the preselected subdomains. In our recent works \cite{ma2022novel,ma2021error}, new optimal local approximations spaces with significant advantages were constructed for the MS-GFEM using new local eigenproblems involving the partition of unity function, and error estimates for the MS-GFEM in the fully discrete setting were established. More recently \cite{ma2021wavenumber}, the method was extended to solve strongly indefinite problems, i.e., heterogeneous Helmholtz problem with high wave numbers, and similar theoretical results were obtained. 

In this paper, we use the MS-GFEM to solve singularly-perturbed reaction-diffusion equations with rough diffusion coefficients at both continuous and discrete levels. The method at the continuous level delivers a multiscale discretization of the problem and at the discrete level, it can be seen as a non-iterative domain decomposition method for solving linear systems resulting from standard FE discretizations of the continuous problem. A rigorous parameter-explicit convergence analysis of the method at both levels is derived, with a focus on local approximation error estimates. The presence of the singular perturbation parameter $\varepsilon$ makes the method behave significantly differently from its usual behavior for classical diffusion problems, and makes the associated analysis much more involved, especially in the discrete setting. On one hand, at the continuous level, singularly perturbed reaction-diffusion equations exhibit a much stronger local property than classical diffusion equations, in the sense that (local) boundary conditions can only affect the solution in a very thin layer of width $O(\varepsilon)$. This local property greatly alleviates the effect of incorrect boundary conditions for the local solution i.e., the solution of the local reaction-diffusion problem, and thus makes it a good local approximation of the exact solution even with a small oversampling size. On the other hand, at the discrete level, the local property is weakened due to spatial discretization, and the performance of the method crucially depends on the relations among $\varepsilon$, the mesh size $h$ of the fine-scale FE discretization, and the oversampling size $\delta^{\ast}$, which requires a careful analysis.

At the continuous level, we prove that the local approximation error in each subdomain in the standard $H^{1}$ norm can be bounded independent of $\varepsilon$ and decays exponentially. More precisely, we prove that it decays exponentially with respect to the local degrees of freedom when the local spectral basis is used, and decays exponentially with respect to $\delta^{\ast}/\varepsilon$ when the local solution alone is used for the local approximation. Therefore, when $\delta^{\ast}/\varepsilon$ is large, the solution of the local problem alone can approximate the exact solution locally very well and the local spectral basis is not needed. At the discrete level, in the asymptotic regime, i.e., $h\leq \varepsilon$, we obtain the same theoretical results as those at the continuous level. However, in the preasymptotic regime $\varepsilon \leq h$, which is of more interest practically, the situation is very different. In this scenario, it is shown that the local approximation error depends critically on $\delta^{\ast}/h$ instead of $\delta^{\ast}/\varepsilon$. In particular, when approximating the fine-scale FE solution locally using only the solution of the (discrete) local problem, we prove that the local error decays exponentially with respect to $\delta^{\ast}/h$, and consequently, a larger oversampling size with respect to $h$ leads to a better local approximation. Furthermore, when the local spectral basis is additionally used, the exponential decay of the local error is proved under a certain assumption involving the quantity $\delta^{\ast}/h$ on the local degrees of freedom.  

The rest of this paper is organized as follows. In \cref{sec:2}, we give the problem formulation and briefly recall the main ingredients and results of GFEM. The MS-GFEM with a detailed convergence analysis at the continuous and discrete levels are presented in \cref{sec:3,sec:4}, respectively. Numerical experiments are carried out in \cref{sec:5} to confirm the theoretical results.

\section{Problem formulation and the GFEM}\label{sec:2}
\subsection{Model problem and standard FE discretizations}
Given $0<\varepsilon \leq 1$, we consider the following reaction-diffusion equation with a heterogeneous diffusion coefficient:
\begin{equation}\label{eq:2-1}
\left\{
\begin{array}{lll}
{\displaystyle -\varepsilon^{2}\nabla\cdot (A\nabla u_{\varepsilon}) + u_{\varepsilon} = f\quad {\rm in}\;\,\, \Omega, }\\[2mm]
{\displaystyle \quad \qquad \qquad \qquad \quad u_{\varepsilon} = 0 \quad {\rm on}\;\,\,\Gamma,}
\end{array}
\right.
\end{equation}
where $\Omega\subset\mathbb{R}^{d}$ ($d=2,3$) is a bounded Lipschitz polyhedral domain with $\Gamma=\partial \Omega$. Throughout this paper, we make the following assumption on the right-hand side $f$ and the coefficient $A$:
\begin{assumption}\label{ass:2-1}
$f\in L^{2}(\Omega)$. $A\in (L^{\infty}(\Omega))^{d\times d}$ is pointwise symmetric, and there exists $0<{a_{\rm min}} < {a_{\rm max}}< +\infty$ such that
\begin{equation}\label{eq:2-2}
{a_{\rm min}} |{\bm \xi}|^{2} \leq A({\bm x}) {\bm \xi}\cdot {\bm \xi} \leq {a_{\rm max}} |{\bm \xi}|^{2}\quad \forall {\bm \xi}\in \mathbb{R}^{d}\quad \forall {\bm x} \in\Omega.
\end{equation}
\end{assumption}
Note that we do not assume any regularity condition other than the usual uniform ellipticity on the diffusion coefficient. Moreover, the method and the corresponding theoretical results in this paper can be extended to the case $f\in H^{-1}(\Omega)$ in a very straightforward way, but we omit this extension for ease of presentation. 

For convenience, let us define the "energy" inner-product and the associated "energy" norm on $H^{1}(D)$ for any subdomain $D\subset \Omega$. For any $u,v\in H^{1}(D)$,
\begin{equation}\label{eq:2-3}
a_{\varepsilon,D}(u,v) := \varepsilon^{2}\int_{{D}} A\nabla u\cdot {\nabla v} \,d{\bm x} + \int_{{D}} uv\,d{\bm x},\quad 
\Vert u\Vert_{a,\varepsilon,{D}} := \sqrt{a_{\varepsilon,{D}}(u,u).}
\end{equation}
If $D=\Omega$, we drop the domain in the subscript and simply write $a_{\varepsilon}(\cdot,\cdot)$ and $\| \cdot\|_{a,\varepsilon}$. The weak formulation of problem \cref{eq:2-1} is: Find $u_{\varepsilon}\in H_{0}^{1}(\Omega)$ such that
\begin{equation}\label{eq:2-4}
{a}_{\varepsilon}(u_{\varepsilon},v) = F(v):=\int_{\Omega}fv\,d{\bm x}\quad \forall v\in H^{1}_{0}(\Omega).
\end{equation}
By \cref{ass:2-1} and the Lax--Milgram theorem, there exists a unique solution $u_{\varepsilon}\in H_{0}^{1}(\Omega)$ to the problem \cref{eq:2-1} with the estimate
\begin{equation}\label{eq:2-5}
\Vert u_{\varepsilon}\Vert_{a,\varepsilon}\leq \Vert f\Vert_{L^{2}(\Omega)}.
\end{equation}


Now we consider a standard FE discretization of the variational problem \cref{eq:2-4}. \mm{Let $\tau_{h} = \{K\}$ be shape-regular triangulations of $\Omega$}. The mesh size $h$ is given by $h:=\max_{K\in\tau_{h}}h_{K}$ with $h_{K}$ denoting the diameter of element $K$, and $h$ is assumed to be sufficiently small to resolve the fine-scale details of the coefficient $A$ and the layer structures of the solution. We now consider, $V_{h}\subset H^{1}(\Omega)$, the standard finite element space of continuous piecewise polynomials of degree $r$ with respect to $\tau_{h}$, and let $V_{h,0} = V_{h}\cap H_{0}^{1}(\Omega)$. The standard Galerkin FEM for problem \cref{eq:2-4} is given by: Find $u_{h,\varepsilon}\in V_{h,0}$ such that
\begin{equation}\label{eq:2-6}
{a}_{\varepsilon}(u_{h,\varepsilon},v_{h}) = F(v_{h})\quad \forall v_{h}\in V_{h,0}.
\end{equation}
In the following, we refer to problem \cref{eq:2-6} as the fine-scale FE problem and its solution as the fine-scale FE solution.

\subsection{GFEM at continuous and discrete levels}
In this subsection we briefly describe the GFEM at both continuous and discrete levels for solving problem \cref{eq:2-1} and its standard FE approximation \cref{eq:2-6}, respectively. Although the classical GFEM was formulated at the continuous level for directly discretizing PDEs, it can be easily adapted in the discrete setting to yield a coarse-space approximation for the fine-scale FE problem. The latter has been widely used in numerical multiscale methods and domain decomposition methods \cite{efendiev2011multiscale,spillane2014abstract}. 

The GFEM starts with an overlapping decomposition of the domain $\Omega$. Let $\{ \omega_{i} \}_{i=1}^{M}$ be a collection of connected open subsets of $\Omega$ satisfying $\cup_{i=1}^{M} \omega_{i} = \Omega$ and a pointwise overlap condition:
\begin{equation}\label{eq:2-6-0}
\exists\,\kappa \in \mathbb{N}\qquad \forall\,{\bm x}\in \Omega\qquad {\rm card}\{i\;|\;{\bm x}\in \omega_{i}\}\leq \kappa.
\end{equation}
Let $\{ \chi_{i} \}_{i=1}^{M}$ be a partition of unity subordinate to the open cover $\{\omega_{i}\}_{i=1}^{M}$ that satisfies:
\begin{equation}\label{eq:2-7}
\begin{array}{lll}
{\displaystyle 0\leq \chi_{i}({\bm x})\leq 1,\quad \sum_{i=1}^{M}\chi_{i}({\bm x}) =1, \quad \forall \,{\bm x}\in \Omega,}\\[4mm]
{\displaystyle \chi_{i}({\bm x})= 0, \quad \forall \,{\bm x}\in \Omega/\omega_{i}, \quad i=1,\cdots,M,}\\[2mm]
{\displaystyle \chi_{i}\in W^{1,\infty}(\omega_{i}),\;\; \|\nabla \chi_{i} \|_{L^{\infty}(\omega_i)} \leq \frac{C_{1}}{{\rm diam}\,(\omega_{i})},\quad i=1,\cdots,M.}
\end{array}
\end{equation}
To proceed, for each $i=1,\cdots,M$, let us define the following local spaces:
\begin{equation}\label{eq:2-7-0}
H^{1}_{\Gamma}(\omega_{i}) = \big\{v\in H^{1}(\omega_{i})\;:\;v = 0\;\;{\rm on}\;\,  \partial\omega_{i} \cap \Gamma\big\},\;\; V_{h,\Gamma}(\omega_i) = \{v_{h}|_{\omega_i}:v_{h}\in V_{h,0} \}.
\end{equation}
Moreover, let $I_{h}:C(\overline{\Omega})\rightarrow V_{h}$ denote the standard Lagrange interpolation operator.

For each $i=1,\cdots,M$, let $S_{n_{i}}(\omega_{i})\subset H^{1}_{\Gamma}(\omega_{i})$ be a local approximation space of dimension $n_{i}$ and $u_{i}^{p}\in H^{1}_{\Gamma}(\omega_{i})$ be a local particular function. The classical GFEM defines the global particular function and the global approximation space by gluing the local components together using the partition of unity:
\begin{equation}\label{eq:2-8}
\begin{array}{lll}
{\displaystyle u^{p} =\sum_{i=1}^{M}\chi_{i}u^{p}_{i},\quad  S_{n}(\Omega) =\Big\{\sum_{i=1}^{M}\chi_{i}\phi_{i}\,:\, \phi_{i}\in S_{n_{i}}(\omega_{i})\Big\}, }
\end{array}
\end{equation}
and then seeks the finite-dimensional Galerkin approximation of problem \cref{eq:2-4} by finding $u_{\varepsilon}^{G}\in u^{p}+S_{n}(\Omega)$ such that
\begin{equation}\label{eq:2-9}
{a}_{\varepsilon}(u_{\varepsilon}^{G}, v) = F(v)\quad \forall v\in S_{n}(\Omega).
\end{equation}

In a FE setting, with (discrete) local approximation spaces $S_{h,n_{i}}(\omega_{i})\subset V_{h,\Gamma}(\omega_{i})$ and (discrete) local particular functions $u_{h,i}^{p}\in V_{h,\Gamma}(\omega_{i})$, the global particular function and the global approximation space are defined in a similar manner:
\begin{equation}\label{eq:2-9-1}
\begin{array}{lll}
{\displaystyle u_{h}^{p} =\sum_{i=1}^{M}I_{h}\big(\chi_{i}u^{p}_{h,i}\big),\quad  S_{h,n}(\Omega) =\Big\{\sum_{i=1}^{M}I_{h}\big(\chi_{i}\phi_{h,i}\big)\,:\, \phi_{h,i}\in S_{h,n_{i}}(\omega_{i})\Big\}. }
\end{array}
\end{equation}
Here the multiplications of the partition of unity functions and the local functions are interpolated into the FE space to ensure a conforming approximation. The GFEM based coarse-space approximation of the fine-scale FE problem \cref{eq:2-6} is defined by: Find $u_{h,\varepsilon}^{G}\in u_{h}^{p}+S_{h,n}(\Omega)$ such that
\begin{equation}\label{eq:2-9-2}
{a}_{\varepsilon}(u_{h,\varepsilon}^{G}, v_{h}) = F(v_{h})\quad \forall v_{h}\in S_{h,n}(\Omega).
\end{equation}

It is clear that the GFEM solutions $u_{\varepsilon}^{G}$ and $u_{h,\varepsilon}^{G}$ are the best approximations of $u_{\varepsilon}$ and $u_{h,\varepsilon}$ in $u^{p}+S_{n}(\Omega)$ and $u_{h}^{p}+S_{h,n}(\Omega)$, respectively, i.e.,
\begin{equation}\label{eq:2-10}
\begin{array}{lll}
{\displaystyle \displaystyle \big\Vert u_{\varepsilon} - u_{\varepsilon}^{G} \big\Vert_{a,\,\varepsilon} = \inf_{\varphi \in u^{p}+S_{n}(\Omega)}\big\Vert u_{\varepsilon} - \varphi \big\Vert_{a,\,\varepsilon},}\\[3mm]
{\displaystyle \displaystyle \big\Vert u_{h,\varepsilon} - u_{h,\varepsilon}^{G} \big\Vert_{a,\,\varepsilon} = \inf_{\varphi_{h} \in u_{h}^{p}+S_{h,n}(\Omega)}\big\Vert u_{h,\varepsilon} - \varphi_{h} \big\Vert_{a,\,\varepsilon}.}
\end{array}
\end{equation}

The following theorem provides the theoretical foundation for the GFEM at both continuous and discrete levels by showing that the global approximation error of the method is determined by local approximation errors. 

\begin{theorem}[\cite{ma2022novel}]\label{thm:2-0}
Let $u\in H^{1}(\Omega)$ and $u_{h}\in V_{h}$. For each $i=1,\cdots,M$, assume that 
\begin{equation}\label{eq:2-11}
\begin{array}{lll}
{\displaystyle \inf_{\varphi_{i}\in u_{i}^{p}+S_{n_i}(\omega_i)}\big\Vert \chi_{i}(u-\varphi_{i})\big\Vert_{a,\,\varepsilon,\,\omega_{i}}\leq e_{i},}\\[3mm]
{\displaystyle  \inf_{\varphi_{h,i}\in u_{h,i}^{p}+S_{h,n_i}(\omega_i)}\big\Vert I_{h}\big(\chi_{i}(u_{h}-\varphi_{h,i})\big)\big\Vert_{a,\,\varepsilon,\,\omega_{i}}\leq \widetilde{e}_{i}.}
\end{array}
\end{equation}
Then,
\begin{equation}\label{eq:2-12}
\begin{array}{lll}
{\displaystyle \inf_{\varphi\in u^{p}+S_{n}(\Omega)} \Vert u - \varphi\Vert_{a,\,\varepsilon} \leq \Big(\kappa\sum_{i=1}^{M}e^{2}_{i}\Big)^{1/2},}\\[3mm] {\displaystyle \inf_{\varphi_{h}\in u_{h}^{p}+S_{h,n}(\Omega)} \Vert u_{h} - \varphi_{h}\Vert_{a,\,\varepsilon} \leq \Big(\kappa\sum_{i=1}^{M}\widetilde{e}^{2}_{i}\Big)^{1/2}.}
\end{array}
\end{equation}
\end{theorem}

Combining \cref{eq:2-10} and \cref{thm:2-0}, we clearly see that the key to achieving a good accuracy for the GFEM is a suitable selection of the local particular functions and the local approximation spaces such that the exact solution or the fine-scale FE solution can be well approximated locally. In MS-GFEM, the local particular functions are defined as solutions of local boundary value problems with artificial boundary conditions and the local approximation spaces are built from eigenfunctions of local eigenproblems defined on generalized harmonic spaces, leading to local approximations with errors decaying nearly exponentially with the local degrees of freedom. The construction of these local approximations for the MS-GFEM adapted to singularly-perturbed reaction-diffusion problems at the continuous and discrete levels are detailed in \cref{sec:3,sec:4}, respectively. 

\mm{Throughout this paper, we will frequently use the following result: Let $D^{\ast}\subset D$ be open connected subsets of $\Omega$ with $\delta={\rm dist}({D},\,\partial {D}^{\ast}\setminus\partial \Omega)>0$. Then, there exists $\eta\in C^{1}(\overline{D^{\ast}})$ such that 
\begin{equation}\label{eq:3-0}
\eta = 0 \;\;{\rm on}\;\; \partial {D}^{\ast}\setminus\partial \Omega; \quad \eta = 1 \;\;{\rm in} \;\;{D}; \quad |\nabla {\eta}| \leq C_{d}/\delta,
\end{equation}
where $C_{d}>0$ only depends on $d$.}

\section{Continuous MS-GFEM}\label{sec:3}
In this section, we shall construct the local particular functions and the local approximation spaces for the MS-GFEM in the continuous setting. Exponential and $\varepsilon$-explicit upper bounds for the local approximation errors are derived. 

\subsection{Local particular functions and local approximation spaces}
\begin{figure}\label{fig:3-1}
\centering
\includegraphics[scale=0.5]{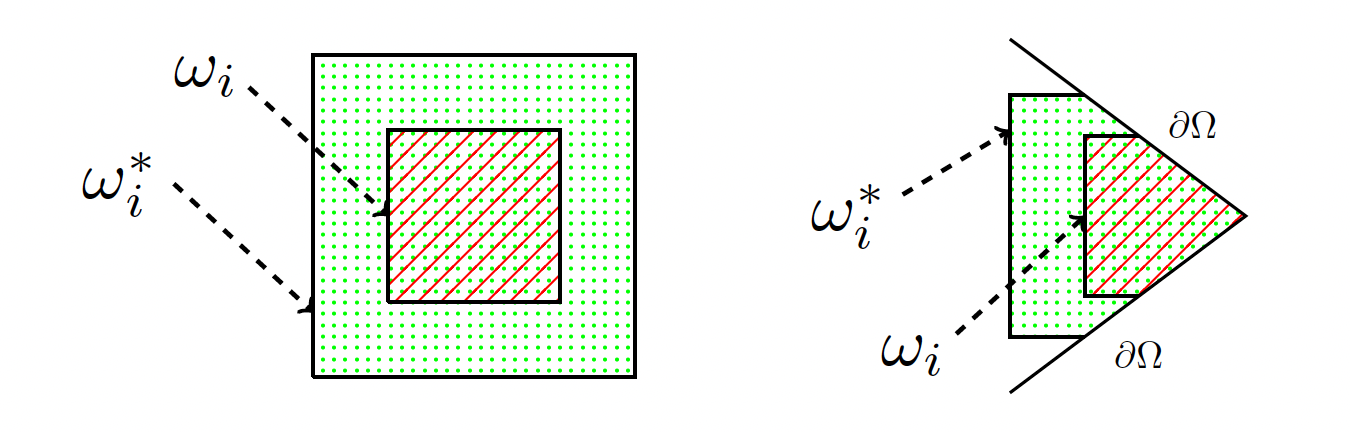}
\caption{Illustration of a subdomain $\omega_i$ that lies within the interior of $\Omega$ (left) and one that intersects the outer boundary $\partial \Omega$ (right) with associated oversampling domains $\omega_i^{\ast}$.}
\end{figure}

A key ingredient of the MS-GFEM is the oversampling technique. More specifically, the local particular functions and the local spectral basis functions are first constructed on a larger domain, often referred to as the oversampling domain, and then restricted to the corresponding subdomain. For each subdomain $\omega_{i}$, we denote by $\omega_{i}^{\ast}$ the associated oversampling domain with a Lipschitz boundary that satisfies $\omega_{i} \subset \omega_{i}^{\ast}\subset\Omega$, as illustrated in \cref{fig:3-1}. As we will see later, the size of these oversampling domains has a crucial effect on the accuracy of local approximations. 

To define the local particular function on a subdomain $\omega_i$, let us consider the following local reaction-diffusion problem on the oversampling domain $\omega_{i}^{\ast}$:
\begin{equation}\label{eq:3-3}
\left\{
\begin{array}{lll}
{\displaystyle -\varepsilon^{2}\nabla \cdot(A\nabla \psi_{\varepsilon,i}) + \psi_{\varepsilon,i}= f \,\qquad {\rm in}\;\, \omega^{\ast}_{i} , }\\[2mm]
{\displaystyle \qquad \quad\qquad\;\;{\bm n}\cdot A\nabla\psi_{\varepsilon,i} = 0\,\qquad {\rm on}\;\, \partial \omega^{\ast}_{i} \cap \Omega,}\\[2mm]
{\;\,\qquad \quad \,\quad\qquad \qquad \displaystyle \psi_{\varepsilon,i} = 0\qquad \,{\rm on}\;\, \partial \omega^{\ast}_{i} \cap \Gamma,}
\end{array}
\right.
\end{equation}
where ${\bm n}$ denotes the unit outward normal. Denoting by
\begin{equation}
H^{1}_{\Gamma}(\omega^{\ast}_{i}) = \big\{v\in H^{1}(\omega^{\ast}_{i})\;:\;v = 0\;\;{\rm on}\;\,  \partial\omega^{\ast}_{i} \cap \Gamma\big\}, 
\end{equation}
the weak formulation of \cref{eq:3-3} is to find $\psi_{\varepsilon,i}\in H^{1}_{\Gamma}(\omega^{\ast}_{i})$ such that
\begin{equation}\label{eq:3-4}
{a}_{\varepsilon,\,\omega^{\ast}_{i}}(\psi_{\varepsilon,i}, v) = F_{\omega_{i}^{\ast}}(v):=\int_{\omega_{i}^{\ast}}fv\,d{\bm x}\qquad \forall v\in H^{1}_{\Gamma}(\omega^{\ast}_{i}).
\end{equation}
It is clear that the weak formulation \cref{eq:3-4} has a unique solution. Now we can define the local particular functions.
\begin{definition}[\textbf{Local particular functions}]\label{def:3-1}
The local particular function on $\omega_{i}$ is defined as $u_{i}^{p} = \psi_{\varepsilon,i}|_{\omega_{i}}$, where $\psi_{\varepsilon,i}\in H^{1}_{\Gamma}(\omega^{\ast}_{i})$ is the solution of \cref{eq:3-3}.
\end{definition}
\begin{rem}
The definition of local particular functions is not unique. Indeed, for the subsequent analysis, it is sufficient that ${a}_{\varepsilon,\,\omega^{\ast}_{i}}(\psi_{\varepsilon,i}, v) = F_{\omega_{i}^{\ast}}(v)$ for all $v\in H^{1}_{0}(\omega^{\ast}_{i})$. Thus, other interior boundary conditions for $\psi_{\varepsilon,i}$ can be imposed. In general, interior boundary conditions that make the local problems better behaved are preferred.
\end{rem}

It turns out that the local particular function $u_{i}^{p}$ is a good approximation of the exact solution $u_{\varepsilon}$ locally on $\omega_{i}$ if $\varepsilon$ is sufficiently small. In fact, we have
\begin{theorem}\label{thm:3-0}
For any $\varepsilon\in (0,1]$, it holds that
\begin{equation}\label{eq:3-9-1}
\big\Vert \chi_{i}(u_{\varepsilon}-u_{i}^{p})\big\Vert_{a,\,\varepsilon,\,\omega_{i}} \leq \varepsilon\,a_{\rm max}^{1/2} \Vert \nabla\chi_{i}\Vert_{L^{\infty}(\omega_{i})} \big(\Vert u_{\varepsilon}\Vert_{L^{2}(\omega_{i})} + \Vert f\Vert_{L^{2}(\omega^{\ast}_{i})}\big). 
\end{equation}
If, in addition, $\delta^{\ast}_{i} := {\rm dist}(\omega_i,\partial\omega_{i}^{\ast}\setminus\partial \Omega)>0$ and $\gamma^{\ast}_{i} := \delta^{\ast}_{i}/(e\varepsilon a_{\rm max}^{1/2}\mm{C_{d}})>1$, \mm{where $C_{d}>0$ is given by \cref{eq:3-0}}, then
\begin{equation}\label{eq:3-9-2}
\begin{array}{lll}
{\displaystyle \big\Vert \chi_{i}(u_{\varepsilon}-u_{i}^{p})\big\Vert_{a,\,\varepsilon,\,\omega_{i}}\leq \varepsilon\,a_{\rm max}^{1/2} \,e^{1-\gamma^{\ast}_{i}} \,\Vert \nabla\chi_{i}\Vert_{L^{\infty}(\omega_{i})} \big(\Vert u_{\varepsilon}\Vert_{L^{2}(\omega_{i}^{\ast})} + \Vert f\Vert_{L^{2}(\omega_{i}^{\ast})}\big). }
\end{array}
\end{equation}
\end{theorem}
Note that \cref{eq:3-9-1} holds in the non-oversampling case, i.e., $\omega^{\ast}_{i} = \omega_{i}$. The proof of \cref{thm:3-0} is postponed to the next subsection.  

If $\varepsilon$ is not sufficiently small, however, $u_{i}^{p}$ may fail to well approximate $u_{\varepsilon}|_{\omega_{i}}$. In this scenario, it is necessary to consider the residual, i.e., $u_{\varepsilon}|_{\omega_{i}}-\psi_{\varepsilon,i}|_{\omega_{i}}$. To approximate this part is the purpose of designing the local approximation space. Taking $v\in H_{0}^{1}(\omega_i^{\ast})$ in \cref{eq:2-4,eq:3-4}, we see that ${a}_{\varepsilon,\,\omega^{\ast}_{i}}(u_{\varepsilon}|_{\omega_{i}^{\ast}}-\psi_{\varepsilon,i}, v)=0$ for all $v\in H^{1}_{0}(\omega^{\ast}_{i})$. This observation motivates us to define the following \textit{generalized harmonic space}:
\begin{equation}\label{eq:3-1}
H_{a,\varepsilon}(\omega^{\ast}_{i}) = \big\{ u\in H^{1}_{\Gamma}(\omega^{\ast}_{i})\;:\;  {a}_{\varepsilon,\,\omega^{\ast}_{i}}(u, v) = 0 \quad \forall v\in H^{1}_{0}(\omega^{\ast}_{i})\big\}.
\end{equation}
It follows that $u_{\varepsilon}|_{\omega_{i}^{\ast}}-\psi_{\varepsilon,i}\in H_{a,\varepsilon}(\omega^{\ast}_{i})$. Note that $H_{a,\varepsilon}(\omega^{\ast}_{i})\subset H^{1}_{\Gamma}(\omega^{\ast}_{i})$ is a closed subspace. It turns out that $H_{a,\varepsilon}(\omega^{\ast}_{i})$ can be well approximated by a low-dimensional space locally on $\omega_i$, enabling us to construct a highly efficient local approximation space. A crucial tool to identify the low-dimensional space is the following Caccioppoli-type inequality, which is also a key to deriving exponential bounds on local approximation errors.
\begin{lemma}\label{lem:3-1}
Assume that $\eta\in W^{1,\infty}(\omega_{i}^{\ast})$ satisfying $\eta =0$ on $\partial \omega^{\ast}_{i} \cap \Omega$. Then,
\begin{equation}\label{eq:3-5}
{a}_{\varepsilon,\omega_{i}^{\ast}}(\eta u,\eta v) = \varepsilon^{2}\int_{\omega_{i}^{\ast}}(A\nabla \eta \cdot \nabla \eta) uv\,d{\bm x}\qquad \forall u, \,v\in H_{a,\varepsilon}(\omega^{\ast}_{i}).
\end{equation}
In particular, 
\begin{equation}\label{eq:3-6}
\Vert \eta u \Vert_{a,\,\varepsilon, \,\omega_{i}^{\ast}} \leq \varepsilon a_{\rm max}^{1/2} \Vert \nabla \eta \Vert_{L^{\infty}(\omega_{i}^{\ast})} \Vert u \Vert_{L^{2}(\omega_{i}^{\ast}\cap\, {\rm supp}\,(\eta))} \quad \forall u\in H_{a,\varepsilon}(\omega^{\ast}_{i}),
\end{equation}
where ${a_{\rm max}}$ is the spectral upper bound of the coefficient $A$ defined in \cref{eq:2-2}.
\end{lemma}

The proof of \cref{lem:3-1} is given in the appendix. Note that the Caccioppoli-type inequality \cref{eq:3-6} holds for $\eta=\chi_{i}$, where $\chi_{i}$ is the partition of unity function supported on $\omega_{i}$. Using this inequality and the Rellich theorem, we see that the operator 
\begin{equation}
P_{i}: \big( H_{a,\varepsilon}(\omega^{\ast}_{i}), \;\Vert \cdot\Vert_{a,\varepsilon,\omega_{i}^{\ast}}\big) \rightarrow \big(H_{0}^{1}(\omega_{i}),\;\Vert \cdot\Vert_{a,\varepsilon,\omega_{i}}\big) \quad {\rm with}\quad P_{i}v = \chi_{i}v
\end{equation}
is compact. To find the low-dimensional space in $H_{a,\varepsilon}(\omega^{\ast}_{i})$, following the idea in \cite{ma2022novel}, we consider the following Kolmogrov $n$-width of the operator $P_{i}$:
\begin{equation}\label{eq:3-10}
d_{n}(\omega_{i},\omega_{i}^{\ast})\mm{:}=\inf_{Q(n)\subset H_{0}^{1}(\omega_{i})}\sup_{u\in H_{a,\varepsilon}(\omega^{\ast}_{i})} \inf_{v\in Q(n)}\frac {\Vert P_{i}u-v\Vert_{a,\varepsilon,\omega_i}}{\Vert u \Vert_{a,\varepsilon,\omega_{i}^{\ast}}},
\end{equation}
where the leftmost infimum is taken over all $n$-dimensional subspaces of $ H_{0}^{1}(\omega_{i})$. The associated optimal approximation space $\widehat{Q}(n)$ satisfies
\begin{equation}\label{eq:3-11}
d_{n}(\omega_{i},\omega_{i}^{\ast}) =\sup_{u\in H_{a,\varepsilon}(\omega^{\ast}_{i})} \inf_{v\in \widehat{Q}(n)}\frac {\Vert P_{i}u-v\Vert_{a,\varepsilon,\omega_i}}{\Vert u \Vert_{a,\varepsilon,\omega_{i}^{\ast}}}.
\end{equation}
Since $P_{i}$ is a compact operator in Hilbert spaces, the associated Kolmogrov $n$-width can be characterized by its singular vectors and singular values; see, e.g., \cite[Theorem 2.5, Chapter 4]{pinkus1985n}. In particular, we have the following characterization of $d_{n}(\omega_{i},\omega_{i}^{\ast})$.
\begin{lemma}\label{lem:3-2}
Let $\{\lambda_{k}\}$ and $\{\phi_{k}\}$ denote the eigenvalues (listed to their multiplicities in non-increasing order) and eigenfunctions of the problem
\begin{equation}\label{eq:3-12}
{a}_{\varepsilon,\omega_{i}}(\chi_{i}\phi, \chi_{i} v) = \lambda\,{a}_{\varepsilon,\omega^{\ast}_{i}}(\phi, v)\quad \forall v\in H_{a,\varepsilon}(\omega^{\ast}_{i}).
\end{equation}
Then, $d_{n}(\omega_{i},\omega_{i}^{\ast}) = \lambda^{1/2}_{n+1}$, and the associated optimal approximation space is given by
\begin{equation}\label{eq:3-13}
\widehat{Q}(n) = {\rm span}\{\chi_{i}\phi_{1},\cdots, \chi_{i}\phi_{n}\}.
\end{equation}
\end{lemma}
\begin{proof}
Let $P_{i}^{\ast}:H_{0}^{1}(\omega_{i})\rightarrow H_{a,\varepsilon}(\omega^{\ast}_{i})$ denote the adjoint of the operator $P_{i}$ in the $a_{\varepsilon,\omega_{i}^{\ast}}(\cdot,\cdot)$ inner-product, and consider the eigenvalue problem of the operator $P_{i}^{\ast}P_{i}$:
\begin{equation}\label{eq:3-14}
    P_{i}^{\ast}P_{i}\phi = \lambda \,\phi.
\end{equation}
Note that the problem \cref{eq:3-12} is the variational formulation of problem \cref{eq:3-14}. Hence, the result follows from \cite[Theorem 2.5, Chapter 4]{pinkus1985n}.
\end{proof}
Now we are ready to define the local approximation spaces for the MS-GFEM.
\begin{definition}[\textbf{Local approximation spaces}]\label{def:3-2}
The local approximation space on $\omega_{i}$ is defined as
\begin{equation}
S_{n_{i}}(\omega_i) =  {\rm span}\big\{{\phi_{1}}|_{\omega_i},\cdots,{\phi_{n_{i}}}|_{\omega_i}\big\},
\end{equation}
where $\phi_{k}$ denotes the $k$-th eigenfunction of \cref{eq:3-12}.
\end{definition}
Combining the definition of the $n$-width and the fact that $u_{\varepsilon}|_{\omega_{i}^{\ast}}-\psi_{\varepsilon,i}\in H_{a,\varepsilon}(\omega^{\ast}_{i})$ gives the following local approximation error estimate.
\begin{theorem}\label{thm:3-1}
For each $i=1,\cdots,M$, let the local particular function $u_{i}^{p}$ and the local approximation space $S_{n_{i}}(\omega_i)$ be defined in \cref{def:3-1,def:3-2}. Then,
\begin{equation}\label{eq:3-15}
\inf_{\varphi_{i} \in u_{i}^{p}+S_{n_{i}}(\omega_i)} \big\Vert \chi_{i}(u_{\varepsilon}-\varphi_{i})\big\Vert_{a,\varepsilon,\omega_{i}}\leq d_{n_{i}}(\omega_{i},\omega_{i}^{\ast})\,\big\Vert u_{\varepsilon}|_{\omega_{i}^{\ast}}-\psi_{\varepsilon,i}\big\Vert_{a, \varepsilon,\omega_{i}^{\ast}},
\end{equation}
where $u_{\varepsilon}$ is the exact solution of \cref{eq:2-4}.
\end{theorem}

\Cref{thm:3-1} shows that the local approximation error on $\omega_{i}$ is essentially bounded by the $n$-width $d_{n_{i}}(\omega_{i},\omega_{i}^{\ast})$. Assuming that $\delta^{\ast}_{i}:={\rm  dist}(\omega_{i},\,\partial \omega_{i}^{\ast}\setminus\partial \Omega)>0$, we can prove an $\varepsilon$-explicit and root exponential upper bound on $d_{n_{i}}(\omega_{i},\omega_{i}^{\ast})$ as follows. 
\begin{theorem}\label{thm:3-2}
There exist $\Lambda_{i}>0$ and $b_{i}>0$ independent of $\varepsilon$, such that for any $\varepsilon\in (0,1]$, 
\begin{equation}\label{eq:3-16-0}
d_{n_{i}}(\omega_{i},\omega_{i}^{\ast}) \leq \varepsilon a_{\rm max}^{1/2} \Vert \nabla\chi_{i}\Vert_{L^{\infty}(\omega_{i})}\,e^{-b_{i}n_{i}^{{1}/{(d+1)}}},\quad \forall \,n_{i}>\Lambda_{i}.
\end{equation}
\end{theorem}
The proof of \cref{thm:3-2} is given in the next subsection where the constants $\Lambda_{i}$ and $b_{i}$ are given explicitly.

\subsection{Local approximation error estimates}\label{sec:3-2}
This subsection is devoted to proving \cref{thm:3-0,thm:3-2}. For ease of notation, the subscript $i$ is omitted in the proof. \mm{Let us recall the constant $C_{d}>0$ given by \cref{eq:3-0}} and start with a sharper Caccioppoli-type inequality as follows.

\begin{lemma}\label{lem:3-4}
Let ${D}\subset{D}^{\ast}$ be open connected subsets of $\Omega$ with $\delta={\rm dist}({D},\,\partial {D}^{\ast}\setminus\partial \Omega)>0$, and let $\eta\in W^{1,\infty}(D^{\ast})$ with ${\rm supp}(\eta)\subset \overline{D}$. If $\delta/(e\varepsilon a_{\rm max}^{1/2}\mm{C_{d}})>1$, then for any $u\in H_{a,\varepsilon}(D^{\ast})$,
\begin{align}
 \Vert u \Vert_{L^{2}({D})} &\leq e^{1-\delta/(e\varepsilon a_{\rm max}^{1/2}\mm{C_{d}})}\;\Vert u \Vert_{L^{2}({D}^{\ast})},\label{eq:3-26}\\[1mm]
\Vert \eta u \Vert_{a,\varepsilon,{D}^{\ast}}&\leq \varepsilon a_{\rm max}^{1/2} e^{1-\delta/(e\varepsilon a_{\rm max}^{1/2}\mm{C_{d}})} \Vert \nabla \eta\Vert_{L^{\infty}(D^{\ast})} \Vert u \Vert_{L^{2}({D}^{\ast})}.\label{eq:3-26-0}
\end{align}
\end{lemma}
\begin{proof}
Let $N\in \mathbb{N}$ and choose $\{{D}_{k}\}_{k=0}^{N}$ such that ${D} = {D}_{0}\subset{D}_{1}\subset\cdots\subset{D}_{N} = {D}^{\ast}$ and ${\rm dist}({D}_{k-1}, \partial {D}_{k}\setminus\partial \Omega) = \delta/N$ for each $1\leq k\leq N$. Next we choose a cut-off function $\widetilde{\eta}\in C^{1}(\overline{{D}_{1}})$ such that 
\begin{equation}\label{eq:3-26-1}
\widetilde{\eta} = 0 \;\;{\rm on}\;\; \partial {D}_{1}\setminus\partial \Omega; \quad \widetilde{\eta} = 1 \;\;{\rm in} \;\;{D}_{0}; \quad |\nabla \widetilde{\eta}| \leq \mm{C_{d}}N/\delta.
\end{equation}
Applying \cref{eq:3-6} to $u\in H_{a,\varepsilon}(D_{1})$ and $\widetilde{\eta}$ satisfying \cref{eq:3-26-1}, we get
\begin{equation}\label{eq:3-27}
\Vert u \Vert_{L^{2}({D}_{0})} \leq \Vert \widetilde{\eta} u \Vert_{a,\varepsilon,{D}_{1}}\leq \big(N\varepsilon a_{\rm max}^{1/2}\mm{C_{d}}/\delta\big) \,\Vert u \Vert_{L^{2}({D}_{1})}.
\end{equation}
Using the above argument recursively on ${D}_{1},\cdots,{D}_{N}$, it follows that
\begin{equation}\label{eq:3-28}
\Vert u \Vert_{L^{2}({D})}=\Vert u \Vert_{L^{2}({D}_{0})} \leq (N\varepsilon a_{\rm max}^{1/2}\mm{C_{d}}/\delta)^{N} \,\Vert u \Vert_{L^{2}({D}^{\ast})}.
\end{equation}
Let $N = \lfloor \delta/ (e\varepsilon a_{\rm max}^{1/2}\mm{C_{d}})\rfloor$. Then we have $N\varepsilon a_{\rm max}^{1/2}\mm{C_{d}}/\delta\leq e^{-1}$ and $N\geq \delta/ (e\varepsilon a_{\rm max}^{1/2}\mm{C_{d}})-1$. It follows from \cref{eq:3-28} that
\begin{equation}\label{eq:3-29}
\Vert u \Vert_{L^{2}({D})}\leq e^{-N} \;\Vert u \Vert_{L^{2}({D}^{\ast})} \leq e^{1-\delta/ (e\varepsilon a_{\rm max}^{1/2}\mm{C_{d}})}\,\Vert u \Vert_{L^{2}({D}^{\ast})},
\end{equation}
which gives \cref{eq:3-26}. To prove \cref{eq:3-26-0}, we first use \cref{eq:3-6} and the assumption that ${\rm supp}(\eta)\subset \overline{D}$ to get 
\begin{equation}\label{eq:3-30}
\Vert \eta u \Vert_{a,\varepsilon,{D}^{\ast}}\leq \varepsilon a_{\rm max}^{1/2}\Vert \nabla \eta\Vert_{L^{\infty}(D^{\ast})} \Vert u\Vert_{L^{2}({D})}.
\end{equation}
Combine \cref{eq:3-30} with \cref{eq:3-26} and the estimate \cref{eq:3-26-0} follows.
\end{proof}

Now we can prove \cref{thm:3-0}.
\begin{proof}[Proof of \cref{thm:3-0}] 
Since $u_{\varepsilon}|_{\omega_{i}^{\ast}}-\psi_{\varepsilon,i}\in H_{a,\varepsilon}(\omega^{\ast}_{i})$ and $\chi_{i} \in W^{1,\infty}(\omega^{\ast}_{i})$ with $\chi_{i}=0$ on $\partial \omega_{i}\cap \Omega$, we can apply the Caccioppoli-type inequality \cref{eq:3-6} to get
\begin{equation}\label{eq:3-30-1}
\Vert \chi_{i}(u_{\varepsilon}-\psi_{\varepsilon,i})\Vert_{a,\varepsilon, \omega_{i}} \leq \varepsilon a_{\rm max}^{1/2} \Vert \nabla \chi_{i} \Vert_{L^{\infty}(\omega_{i})} \Vert u_{\varepsilon}-\psi_{\varepsilon,i} \Vert_{L^{2}(\omega_{i})}.
\end{equation}
Noting that $\Vert\psi_{\varepsilon,i}\Vert_{L^{2}(\omega_{i})}\leq \Vert\psi_{\varepsilon,i}\Vert_{L^{2}(\omega_{i}^{\ast})}\leq \Vert f\Vert_{L^{2}(\omega_{i}^{\ast})}$, \cref{eq:3-9-1} follows immediately from \cref{eq:3-30-1}. The second part of \cref{thm:3-0} can be proved in a similar way by applying \cref{eq:3-26-0} on $\omega_{i}^{\ast}$.
\end{proof}

In the rest of this subsection, we prove \cref{thm:3-2} based on an explicit construction of a subspace of $H_{0}^{1}(\omega)$ with approximation error decaying nearly exponentially. To this end, we first give some useful lemmas.

\begin{lemma}\label{lem:3-3-0}
Let ${D}$ and ${D}^{\ast}$ be open connected subsets of $\Omega$ with ${D}\subset{D}^{\ast}$ and $\delta = {\rm dist}({D},\,\partial {D}^{\ast}\setminus\partial \Omega)>0$. There exist positive constants $C_{1}$ and $C_{2}$ depending only on $d$, such that for each integer $m\geq C_{1}|D^{\ast}|\delta^{-d}$, there exists an $m$-dimensional space $V_{m}({D}^{\ast})\subset L^{2}(D^{\ast})$ satisfying
\begin{equation}\label{eq:3-16-1}
\inf_{v\in V_{m}({D}^{\ast})} \Vert u-v\Vert_{L^{2}({D})}\leq C_{2}m^{-1/d} |D^{\ast}|^{1/d} \Vert \nabla u \Vert_{L^{2}({D}^{\ast})}\quad \forall u\in H^{1}(D^{\ast}).
\end{equation}
\end{lemma}
\begin{proof}
\mm{First of all, let us fix a quasi-uniform family of triangulations $\{ \mathcal{T}_{H}\}_{H>0}$ of $\Omega$ with $\max_{T\in \mathcal{T}_{H}} h_{T} = H \lesssim \min_{T\in \mathcal{T}_{H}} h_{T}$, constructed by successively refining an arbitrary initial mesh. For any open connected subsets ${D}\subset{D}^{\ast}$ of $\Omega$ with $\delta = {\rm dist}({D},\,\partial {D}^{\ast}\setminus\partial \Omega)>0$, we can select a triangulation $\mathcal{T}_{H}$ constructed above with $0<H\leq\delta$.} Let $\widetilde{\mathcal{T}}_{H}$ denote the collection of elements in $\mathcal{T}_{H}$ that intersect $D$, i.e.,
\begin{equation}
    \widetilde{\mathcal{T}}_{H}  = \big\{T\in \mathcal{T}_{H}:T\cap D\neq \emptyset \big\},
\end{equation}
and \mm{let} $\widetilde{D}_{H}$ denote the domain made of the elements in $\widetilde{\mathcal{T}}_{H}$. Since $0<H\leq\delta$, we see that $D\subset \widetilde{D}_{H} \subset {D}^{\ast}$. Denote by $\mathcal{N}$ the set of vertices of $\widetilde{\mathcal{T}}_{H}$ and by $\{g_{N}|N\in \mathcal{N}\}$ the corresponding basis of hat functions. Let $m=\#\mathcal{N}$. We define the desired approximation space $V_{m}({D}^{\ast})\subset L^{2}(D^{\ast})$ as 
\begin{equation}
V_{m}({D}^{\ast}) = {\rm span} \big\{\widetilde{g}_{N}: \widetilde{g}_{N}|_{\widetilde{D}_{H}} = g_{N},\; \widetilde{g}_{N}|_{{D}^{\ast}\setminus\widetilde{D}_{H}} = 0 \quad \forall N\in \mathcal{N}\big\}.    
\end{equation}
Using the approximation property of the Cl\'{e}ment interpolation \cite{clement1975approximation} and the fact that $D\subset \widetilde{D}_{H} \subset {D}^{\ast}$, we see that for any $u\in H^{1}(D^{\ast})$,
\begin{equation}\label{eq:3-16}
\begin{array}{lll}
{\displaystyle \inf_{v\in V_{m}({D}^{\ast})} \Vert u-v\Vert_{L^{2}(D)}\leq \inf_{v\in \mm{V_{m}({D}^{\ast})}} \Vert u-v\Vert_{L^{2}(\widetilde{D}_{H})} }\\[4mm] {\displaystyle\qquad  \leq C_{d,0}H\Vert \nabla u\Vert_{L^{2}(\widetilde{D}_{H})} \leq C_{d,0}H \Vert \nabla u\Vert_{L^{2}(D^{\ast})} }.  
\end{array}
\end{equation}
Moreover, by the quasi uniformity of the mesh, we have $C_{d,1} |\widetilde{D}_{H}|\leq mH^{d}\leq C_{d,2}|\widetilde{D}_{H}|$, which, combining with \cref{eq:3-16}, gives that
\begin{equation}
\begin{array}{lll}
{\displaystyle
 \inf_{v\in V_{m}({D}^{\ast})} \Vert u-v\Vert_{L^{2}(D)}\leq C_{d,0}C_{d,2}^{1/d}m^{-1/d} |\widetilde{D}_{H}|^{1/d} \Vert \nabla u\Vert_{L^{2}(D^{\ast})} }\\[4mm]
 {\displaystyle \qquad \leq C_{d,0}C_{d,2}^{1/d} m^{-1/d} |D^{\ast}|^{1/d} \Vert \nabla u\Vert_{L^{2}(D^{\ast})}.}  
\end{array}
\end{equation}
\mm{Finally, to guarantee the existence of an $H$ satisfying $C_{d,1} |\widetilde{D}_{H}|\leq mH^{d}\leq C_{d,2}|\widetilde{D}_{H}|$, $0<H\leq \delta$, and that $\mathcal{T}_{H}\in \{ \mathcal{T}_{H}\}$, it is sufficient that $m\geq C_{d,1}|D^{\ast}|(\delta/2)^{-d}$ (slightly increasing $C_{d,2}$ if necessary).}
\end{proof}
\begin{rem}
The $L^{2}$ approximation error in \cref{eq:3-16-1} is only estimated in the interior of $D^{\ast}$ as here we do not impose any regularity conditions on the boundary of $D^{\ast}$ (except on the part $\partial D^{\ast}\cap \partial \Omega)$, i.e., $\partial D^{\ast}\cap \Omega$ can be very rough. If $D^{\ast}$ is a sufficiently regular domain, e.g., a sphere, cube, or tetrahedron, we can obtain the $L^{2}$ approximation error in the whole of $D^{\ast}$ without any assumption on $m$. More generally, if $D^{\ast}$ has a Lipschitz boundary so that the embedding $H^{1}(D^{\ast})\subset L^{2}(D^{\ast})$ is compact, then there exist $m$-dimensional spaces $V_{m}({D}^{\ast})\subset L^{2}(D^{\ast})$ such that
\begin{equation}
\inf_{v\in V_{m}({D}^{\ast})} \Vert u-v\Vert_{L^{2}({D}^{\ast})}\leq \lambda^{-1/2}_{m+1} \,\Vert \nabla u \Vert_{L^{2}({D}^{\ast})}\quad \forall u\in H^{1}(D^{\ast}),
\end{equation}
where $\lambda_{m+1}$ denotes the $(m+1)$-th eigenvalue of the Laplace operator on \mm{${D}^{\ast}$} with the Neumann boundary condition on \mm{$\partial {D}^{\ast}$}. In \cite{solomyak1980quantitative,netrusov2005weyl}, it was proved that $\lambda^{-1/2}_{m+1} \leq Cm^{-1/d} |D^{\ast}|^{1/d}$, where $C>0$ may depend on the "roughness" of the boundary $\partial D^{\ast}$. Indeed, this approximation result can be proved by using a similar technique as in \cref{lem:3-3-0} (interpolation error estimates on a mesh covering the domain $D^{\ast}$) and the extension property for Lipschitz domains (see, e.g., \cite[Theorem 5, Chapter VI]{stein1970singular}). Finally, we note that \cref{lem:3-3-0} can also be proved for the case where $\partial D^{\ast}\cap \partial \Omega$ is $C^{1}$ smooth, using an extension technique as in \cite{babuska2011optimal} and a similar argument as above. In this case, the constants $C_{1}$ and $C_{2}$ may depend on $\partial D^{\ast}\cap \partial \Omega$.
\end{rem}

The following lemma shows that the approximation result in \cref{lem:3-3-0} can be extended to any closed subspace of $H^{1}(D^{\ast})$. The key point is that the approximation space is required to be in the given subspace.

\begin{lemma}\label{lem:3-3}
Let ${D}$ and ${D}^{\ast}$ be open connected subsets of $\Omega$ with ${D}\subset{D}^{\ast}$ and $\delta = {\rm dist}({D},\,\partial {D}^{\ast}\setminus\partial \Omega)>0$, and let $\mathcal{S}(D^{\ast})$ be a closed subspace of $H^{1}(D^{\ast})$. In addition, let \mm{the} constants $C_{1}$ and $C_{2}$ be as in \cref{lem:3-3-0}. Then, for each integer $m\geq C_{1}|D^{\ast}|\delta^{-d}$, there exists an $m$-dimensional space $\Psi_{m}({D}^{\ast})\subset \mathcal{S}(D^{\ast})$ such that
\begin{equation}\label{eq:3-17}
\inf_{\varphi\in \Psi_{m}({D}^{\ast})} \Vert u-\varphi\Vert_{L^{2}({D})}\leq C_{2}m^{-1/d} |D^{\ast}|^{1/d} \rmm{\Vert u \Vert_{H^{1}({D}^{\ast})}}\quad \forall u\in \mathcal{S}(D^{\ast}).
\end{equation}
\rmm{In addition, if the $H^{1}$-seminorm $\Vert \nabla \cdot\Vert_{L^{2}(D^{\ast})}$ is a norm on $\mathcal{S}(D^{\ast})$ equivalent to the standard $H^{1}$-norm, then $\Psi_{m}({D}^{\ast})\subset \mathcal{S}(D^{\ast})$ can be chosen such that
\begin{equation}\label{eq:3-18}
\inf_{\varphi\in \Psi_{m}({D}^{\ast})} \Vert u-\varphi\Vert_{L^{2}({D})}\leq C_{2}m^{-1/d} |D^{\ast}|^{1/d} \Vert \nabla u \Vert_{L^{2}({D}^{\ast})}\quad \forall u\in \mathcal{S}(D^{\ast}).
\end{equation}}
\end{lemma}
\begin{proof}
\rmm{Since $\mathcal{S}(D^{\ast})$ is a closed subspace of $H^{1}(D^{\ast})$, we see that $(\mathcal{S}(D^{\ast}), \Vert\cdot\Vert_{H^{1}(D^{\ast})})$ is a Hilbert space. Now let us consider the following $n$-width:
\begin{equation}\label{eq:3-19}
d_{m}(R)=\inf_{Q(m)\subset L^{2}({D})}\sup_{u\in \mathcal{S}(D^{\ast})} \inf_{v\in Q(m)}\frac {\Vert R(u)-v\Vert_{L^{2}({D})}}{\Vert{u}\Vert_{H^{1}(D^{\ast})}},
\end{equation}
where $R: \mathcal{S}(D^{\ast})\rightarrow L^{2}(D)$ is the restriction operator defined as $R(u) = u|_{D}$. Since $\mathcal{S}(D^{\ast})\subset H^{1}(D^{\ast})$, we can deduce from \cref{lem:3-3-0} that for each $m\geq C_{1}|D^{\ast}|\delta^{-d}$, 
\begin{equation}\label{eq:3-20}
d_{m}(R)\leq C_{2}m^{-1/d} |D^{\ast}|^{1/d}.
\end{equation}
Hence, $d_{m}(R)\rightarrow 0$ as $m\rightarrow \infty$. It follows that the operator $R$ is compact (see, e.g., \cite[Proposition 7.1]{pinkus1985n}). For $j\in \mathbb{N}$, let $\psi_{j}$ denote the $j$-th eigenvector of the problem:
\begin{equation}\label{eq:3-21}
(\psi,\,v)_{L^{2}({D})} = \lambda\,(\psi,\,v)_{H^{1}(D^{\ast})} \quad \forall v\in \mathcal{S}({D}^{\ast}),
\end{equation}
and define $\Psi_{m}({D}^{\ast}) = {\rm span}\{\psi_{1},\ldots,\psi_{m}\} \subset \mathcal{S}({D}^{\ast})$. Using the characterization of $n$-widths in Hilbert spaces (see, e.g., \cite[Theorem 2.5]{pinkus1985n}), we see that the range of $R$ restricted to $\Psi_{m}({D}^{\ast})$ is an optimal approximation space associated with $d_{m}(R)$. Combining this fact with \cref{eq:3-20} yields that for each $m\geq C_{1}|D^{\ast}|\delta^{-d}$,
\begin{equation}\label{eq:3-22}
\inf_{\varphi\in \Psi_{m}({D}^{\ast})} \Vert u-\varphi\Vert_{L^{2}({D})}\leq  C_{2}m^{-1/d} |D^{\ast}|^{1/d} \Vert u \Vert_{H^{1}({D}^{\ast})}\quad \forall u\in \mathcal{S}(D^{\ast}).
\end{equation}
If the $H^{1}$-seminorm $\Vert \nabla \cdot\Vert_{L^{2}(D^{\ast})}$ is equivalent to the norm $\Vert \cdot\Vert_{H^{1}(D^{\ast})}$ on $\mathcal{S}(D^{\ast})$, then $(\mathcal{S}(D^{\ast}), \Vert\nabla \cdot\Vert_{L^{2}(D^{\ast})})$ is a Hilbert space. By modifying the definition of the $n$-width $d_{m}(R)$ accordingly and proceeding in a similar way as above, we get \cref{eq:3-18}.  }
\end{proof}

A combination of \cref{lem:3-3} and the Caccioppoli-type inequality in \cref{lem:3-1} gives the following approximation result in the energy norm.
\begin{lemma}\label{lem:3-5}
Let ${D}$ and ${D}^{\ast}$ be open connected subsets of $\Omega$ with ${D}\subset{D}^{\ast}$ and $\delta = {\rm dist}({D},\,\partial {D}^{\ast}\setminus\partial \Omega)>0$, and let \mm{the} constants $C_{1}$ and $C_{2}$ be as in \cref{lem:3-3-0}. Then, for each integer $m\geq C_{1}|D^{\ast}|(\delta/2)^{-d}$, there exits an $m$-dimensional space $\Psi_{m}({D}^{\ast})\subset H_{a,\varepsilon}({D}^{\ast})$ such that for any $u\in H_{a,\varepsilon}({D}^{\ast})$,
\begin{equation}\label{eq:3-33}
\inf_{\varphi\in \Psi_{m}({D}^{\ast})} \Vert u-\varphi\Vert_{a,\varepsilon,{D}}\leq 2\mm{C_{d}}C_{2} \Big(\frac{{a_{\rm max}}}{{a_{\rm min}}}\Big)^{1/2}|{D}^{\ast}|^{1/d}m^{-1/d}\delta^{-1}\,\Vert u \Vert_{a,\varepsilon,{D}^{\ast}},
\end{equation}
\mm{where $C_{d}>0$ is given by \cref{eq:3-0}}.
\end{lemma}
\begin{proof}
Denote by $D_{\delta/2}$ the open connected subset of $\Omega$ satisfying ${D}\subset D_{\delta/2}\subset{D}^{\ast}$ and ${\rm dist}({D},\,\partial {D}_{\delta/2}\setminus\partial \Omega) = {\rm dist}({D}_{\delta/2},\,\partial {D}^{\ast}\setminus\partial \Omega) = \delta/2$. Note that $H_{a,\varepsilon}({D}^{\ast})$ is a closed subspace of $H^{1}(D^{\ast})$. We first apply \cref{lem:3-3} on $D_{\delta/2}$ and $D^{\ast}$ to deduce that for each $m\geq C_{1}|D^{\ast}|(\delta/2)^{-d}$, there exists an $m$-dimensional space $\Psi_{m}({D}^{\ast})\subset H_{a,\varepsilon}({D}^{\ast})$ such that for any $u\in H_{a,\varepsilon}({D}^{\ast})$,
\begin{equation}\label{eq:3-34}
\begin{array}{lll}
{\displaystyle \inf_{\varphi\in \Psi_{m}({D}^{\ast})} \Vert u-\varphi\Vert_{L^{2}(D_{\delta/2})}\leq C_{2}m^{-1/d} |D^{\ast}|^{1/d} \rmm{\big(\Vert \nabla u \Vert^{2}_{L^{2}({D}^{\ast})} + \Vert u \Vert^{2}_{L^{2}({D}^{\ast})}\big)^{1/2}} }\\[4mm]
{\displaystyle \qquad \qquad \qquad \qquad \qquad \quad \;\;\leq C_{2} a_{\rm min}^{-1/2} \varepsilon^{-1} m^{-1/d} |D^{\ast}|^{1/d} \,\Vert u \Vert_{a,\varepsilon,{D}^{\ast}}.}
\end{array}
\end{equation}
\rmm{Here we have assumed that $\varepsilon^{-1}a^{-1/2}_{\rm min}\geq 1$ without loss of generality.} To proceed, we choose a cut-off function $\eta\in C^{1}(\overline{{D}_{\delta/2}})$ with $\eta = 0$ on $\partial {D}_{\delta/2}\cap\Omega$, $\eta =1$ in ${D}$, and $|\nabla \eta|\leq 2\mm{C_{d}}/\delta$. Combining \cref{eq:3-6} and \cref{eq:3-34}, we see that for any $u\in H_{a,\varepsilon}({D}^{\ast})$,
\begin{equation}\label{eq:3-34-0}
\begin{array}{lll}
{\displaystyle \inf_{\varphi\in \Psi_{m}({D}^{\ast})} \Vert \eta(u-\varphi)\Vert_{a,\varepsilon,{D}_{\delta/2}} \leq \big(2\mm{C_{d}}\varepsilon a_{\rm max}^{1/2}/\delta\big) \inf_{\varphi\in \Psi_{m}({D}^{\ast})} \Vert u-\varphi\Vert_{L^{2}(D_{\delta/2})} }\\[3mm]
{\displaystyle \qquad \quad \leq 2\mm{C_{d}}C_{2} \big(a_{\rm max}/a_{\rm min}\big)^{1/2} \delta^{-1} m^{-1/d} |D^{\ast}|^{1/d} \Vert u \Vert_{a,\varepsilon,{D}^{\ast}}.}
\end{array}
\end{equation}
The desired estimate \cref{eq:3-33} follows immediately by noting that $\eta =1$ in ${D}$.
\end{proof}


\Cref{lem:3-5} shows that any generalized harmonic function can be approximated in a finite-dimensional space in the energy norm restricted to a subdomain with an algebraic convergence rate. In what follows, we apply \cref{lem:3-5} to a family of nested domains between the subdomain $\omega$ and the oversampling domain $\omega^{\ast}$ to obtain a finite-dimensional approximation space with a superalgebraic convergence rate. Recall that $\delta^{\ast} = {\rm dist}(\omega,\,\partial \omega^{\ast}\setminus\partial \Omega)$. For any integer $N\geq 1$, denote by $\{\omega^{j}\}_{j=1}^{N+1}$ a family of nested domains with $\omega=\omega^{N+1}\subset \omega^{N}\subset\cdots\subset\omega^{1} = \omega^{\ast}$ and ${\rm dist}(\omega^{k},\partial \omega^{k+1}\setminus\partial \Omega) = \delta^{\ast}/N$. Let $n=N\times m$. We define 
\begin{equation}\label{eq:3-35}
\Psi(n,\omega,\omega^{\ast}) = \Psi_{m}(\omega^{1})+ \cdots+\Psi_{m}(\omega^{N}).
\end{equation}

A repeated application of \cref{lem:3-5} on the nested domains gives the following approximation result for the space $\Psi(n,\omega,\omega^{\ast})$.
\begin{lemma}\label{lem:3-6}
Let $\chi$ be the partition of unity function supported on $\omega$, and let the constants $C_{1}$ and $C_{2}$ be as in \cref{lem:3-3-0}. In addition, let $m$ and $N$ satisfy $m\geq C_{1} |\omega^{\ast}| (2N/\delta^{\ast})^{d}$. Then, for any $u\in H_{a,\varepsilon}(\omega^{\ast})$, 
\begin{equation}\label{eq:3-36}
\inf_{\varphi \in \Psi(n,\omega,\omega^{\ast})} \Vert \chi(u-\varphi)\Vert_{a,\varepsilon,\omega} \leq \varepsilon a_{\rm max}^{1/2}\Vert \nabla\chi\Vert_{L^{\infty}(\omega)} \xi^{N} \Vert u \Vert_{a,\varepsilon,\omega^{\ast}},
\end{equation}
where $\xi$ is given by
\begin{equation}\label{eq:3-37}
\xi = \xi(m,N)=2\mm{C_{d}}C_{2}Nm^{-1/d}\Big(\frac{{a_{\rm max}}}{{a_{\rm min}}}\Big)^{1/2}\frac{|\omega^{\ast}|^{1/d}}{\delta^{\ast}}.
\end{equation}
\end{lemma}
\begin{proof}
Following the lines of the proof of \cite[Lemma 3.13]{ma2022novel}, we can use \cref{lem:3-5} recursively to find 
%
$\varphi_{u}^{k}\in \Psi_{m}(\omega^{k})$, $k=1,\cdots,N$, such that 
\begin{equation}\label{eq:3-40}
\begin{array}{lll}
{\displaystyle \Big \Vert u - \sum_{k=1}^{N}\varphi_{u}^{k}\Big\Vert_{a,\varepsilon,\omega}  \leq \xi^{N} \Vert u \Vert_{a,\varepsilon,\omega^{\ast}}. }
\end{array}
\end{equation}
Finally, noting that ${\rm supp}(\chi)\subset \overline{\omega}$, we apply the Caccioppoli-type inequality \cref{eq:3-6} to $u - \sum_{k=1}^{N}\varphi_{u}^{k}\in H_{a,\varepsilon}(\omega)$ and $\eta =\chi$ and combine the result with \cref{eq:3-40}. It follows that
\begin{equation}\label{eq:3-41}
\begin{array}{lll}
{\displaystyle \Big \Vert \chi \Big(u - \sum_{k=1}^{N}\varphi_{u}^{k}\Big)\Big\Vert_{a,\varepsilon,\omega} \leq \varepsilon a_{\rm max}^{1/2} \Vert \nabla\chi\Vert_{L^{\infty}(\omega)} \Big \Vert u - \sum_{k=1}^{N}\varphi_{u}^{k}\Big\Vert_{a,\varepsilon,\omega}}\\[3mm]
{\displaystyle \qquad \qquad \qquad \qquad \qquad \leq \varepsilon a_{\rm max}^{1/2} \Vert \nabla\chi\Vert_{L^{\infty}(\omega)} \xi^{N} \Vert u \Vert_{a,\varepsilon,\omega^{\ast}},}
\end{array}
\end{equation}
which yields \cref{eq:3-37}.

\end{proof}

Having established a superalgebraic error bound for the approximation space $\Psi(n,\omega,\omega^{\ast})$, we are ready to prove \cref{thm:3-2}.
\begin{proof}[Proof of \cref{thm:3-2}] Let $Q(n) := \big\{\chi u: u\in \Psi(n,\omega,\omega^{\ast})\big\}\subset H_{0}^{1}(\omega)$. By the definition of the $n$-width and \cref{lem:3-6}, we see that
\begin{equation}\label{eq:3-44}
\begin{array}{lll}
{\displaystyle d_{n}(\omega,\omega^{\ast})\leq \sup_{u\in H_{a,\varepsilon}(\omega^{\ast})}\inf_{\varphi \in Q(n)} \frac{\Vert \chi u-\varphi\Vert_{a,\varepsilon,\omega}}{\Vert u \Vert_{a,\varepsilon,\omega^{\ast}}} \leq \varepsilon a_{\rm max}^{1/2} \Vert \nabla\chi\Vert_{L^{\infty}(\omega)} \xi^{N}.}
\end{array}
\end{equation}
Let $\Theta = 2C\mm{C_{d}}\big(a_{\rm max}/a_{\rm min}\big)^{1/2}|\omega^{\ast}|^{1/d}/\delta^{\ast}$, where $C=\max\{C_{1},C_{2}\}$ with $C_{1}$ and $C_{2}$ given by \cref{lem:3-3-0}, and define
\begin{equation}\label{eq:3-44-0}
\Lambda= 2(4e\Theta)^{d},\; b= \big(2e\Theta+1/2\big)^{-d/(d+1)}.
\end{equation} 
Following the same lines of the proof of \cite[theorem 3.5]{ma2022novel} by choosing $m$ such that
\begin{equation}\label{eq:3-44-1}
\big(e\Theta(N+1)^{2}/N\big)^{d}\leq m < \big(1+e\Theta(N+1)^{2}/N\big)^{d},
\end{equation}
it can be shown that for any $n=Nm>\Lambda$, 
\begin{equation}\label{eq:3-45}
\xi^{N}\leq e^{-bn^{{1}/{(d+1)}}}.
\end{equation}
Inserting \cref{eq:3-45} into \cref{eq:3-44} yields \cref{eq:3-16-0}, and the proof of \cref{thm:3-2} is complete.
\end{proof}


\subsection{Global approximation error estimates} In this subsection, we collect the local approximation error estimates proved in the preceding subsection to derive global error bounds for the method. We will distinguish two cases: (i) only the local particular functions are used for the local approximations; (ii) the local approximation spaces $S_{n_i}(\omega_i)$ defined in \cref{def:3-2} are used.

In order to derive the global error estimates, we assume that the oversampling domains $\{\omega_{i}^{\ast}\}_{i=1}^{M}$ satisfy a similar pointwise overlap condition as $\{\omega_{i}\}_{i=1}^{M}$:
\begin{equation}\label{eq:3-46-0}
\exists\,\kappa^{\ast} \in \mathbb{N}\qquad \forall\,{\bm x}\in \Omega\qquad {\rm card}\{i\;|\;{\bm x}\in \omega^{\ast}_{i}\}\leq \kappa^{\ast}.
\end{equation}
For convenience, let us define the following constants: 
\begin{equation}\label{eq:3-47}
C_{\chi} = \max_{i=1,\cdots,M} \Vert \nabla\chi_{i}\Vert_{L^{\infty}(\omega_{i})},\quad \delta^{\ast}_{i} = {\rm dist}(\omega_i,\,\partial\omega_{i}^{\ast}\setminus\partial \Omega)\;\big(i=1,\cdots,M\big).
\end{equation}
The following lemma gives the global error estimates for the method when no local approximation spaces are used.
\begin{lemma}\label{lem:3-7}
Let $u_{\varepsilon}$ be the exact solution of \cref{eq:2-4} and let $u^{G}_{\varepsilon} = u^{p}$ be the GFEM approximation, where $u^{p}$ denotes the global particular function defined by \cref{eq:2-8}. Then,
\begin{equation}\label{eq:3-48}
\big\Vert u_{\varepsilon} - u^{G}_{\varepsilon} \big \Vert_{a, \varepsilon}=\big\Vert u_{\varepsilon} - u^{p} \big \Vert_{a, \varepsilon} \leq 2\varepsilon \,\sqrt{\kappa\kappa^{\ast}} a_{\rm max}^{1/2}C_{\chi} \Vert f\Vert_{L^{2}(\Omega)}.
\end{equation}
If, in addition, for each $i=1,\cdots,M$, $\gamma^{\ast}_{i} := \delta^{\ast}_{i}/(e\varepsilon a_{\rm max}^{1/2}\mm{C_{d}})>1$, then
\begin{equation}\label{eq:3-49}
\begin{array}{lll}
{\displaystyle \big\Vert u_{\varepsilon} - u^{G}_{\varepsilon} \big \Vert_{a, \varepsilon} \leq 2e\varepsilon\,\sqrt{\kappa \kappa^{\ast}}a_{\rm max}^{1/2} C_{\chi}\big(\max_{i=1,\cdots,M}e^{-\gamma^{\ast}_{i}}\big) \,\Vert f\Vert_{L^{2}(\Omega)}.}
\end{array}
\end{equation}
\end{lemma}
\begin{proof}
Combining \cref{thm:2-0} and \cref{eq:3-9-1}, we see that
\begin{equation}
\begin{array}{lll}
{\displaystyle \big\Vert u_{\varepsilon} - u^{p} \big \Vert^{2}_{a,\varepsilon} \leq \varepsilon^{2}\kappa\, {a_{\rm max}} \sum_{i=1}^{M} \Vert \nabla\chi_{i}\Vert^{2}_{L^{\infty}(\omega_{i})} \big(\Vert u_{\varepsilon}\Vert_{L^{2}(\omega^{\ast}_{i})} + \Vert f\Vert_{L^{2}(\omega^{\ast}_{i})}\big)^{2} }\\[4mm]
{\displaystyle \qquad \qquad \qquad \;\leq 2\varepsilon^{2}C_{\chi}^{2} \kappa\, {a_{\rm max}} \sum_{i=1}^{M} \big(\Vert u_{\varepsilon}\Vert^{2}_{L^{2}(\omega_{i})} + \Vert f\Vert^{2}_{L^{2}(\omega_{i})}\big).}
\end{array}
\end{equation}
Using the pointwise overlap condition \cref{eq:3-46-0} and the estimate \cref{eq:2-5}, we further have
\begin{equation}
\begin{array}{lll}
{\displaystyle \big\Vert u_{\varepsilon} - u^{p} \big \Vert^{2}_{a,\varepsilon} \leq 2\varepsilon^{2} C_{\chi}^{2} \kappa\kappa^{\ast}{a_{\rm max}} \big(\Vert u_{\varepsilon}\Vert^{2}_{L^{2}(\Omega)} + \Vert f\Vert^{2}_{L^{2}(\Omega)}\big) }\\[3mm]
{\displaystyle \qquad \qquad \qquad \,\leq 4\varepsilon^{2} C_{\chi}^{2} \kappa\kappa^{\ast}{a_{\rm max}}\Vert f\Vert^{2}_{L^{2}(\Omega)},}
\end{array}
\end{equation}
which yields \cref{eq:3-48}. The estimate \cref{eq:3-49} can be proved similarly by using \cref{eq:3-9-2}.
\end{proof}

When the local approximation spaces defined in \cref{def:3-2} are used, we can apply \cref{thm:3-2} and a similar argument as above to prove
\begin{lemma}\label{lem:3-8}
Let $u_{\varepsilon}$ be the exact solution of \cref{eq:2-4} and let $u^{G}_{\varepsilon}$ be the GFEM approximation. In addition, let $S_{n_{i}}(\omega_{i})$ be defined in \cref{def:3-2}, and let $\delta_{i}^{\ast}>0$. If for each $i=1,\cdots,M$, $n_{i}>\Lambda_{i}$, then
\begin{equation}
\big\Vert u_{\varepsilon} - u^{G}_{\varepsilon} \big \Vert_{a,\varepsilon} \leq 2\varepsilon \,\sqrt{\kappa\kappa^{\ast}} a_{\rm max}^{1/2}C_{\chi}  \big(\max_{i=1,\cdots,M} e^{-b_{i}n_{i}^{{1}/{(d+1)}}}\big) \Vert f\Vert_{L^{2}(\Omega)},
\end{equation}
where $\Lambda_{i}$ and $b_{i}$ are positive constants.
\end{lemma}

It is important to note that the standard $H^{1}(\Omega)$-norm of $u_{\varepsilon}-u_{\varepsilon}^{G}$, as implied by \cref{lem:3-7,lem:3-8}, can be bounded by a constant independent of $\varepsilon$.

\section{Discrete MS-GFEM}\label{sec:4}
In this section, in the same spirit as the continuous MS-GFEM, local particular functions and local approximation spaces are constructed at the discrete level to approximate the fine-scale FE solution $u_{h,\varepsilon}$ locally, giving rise to the discrete MS-GFEM. The focus of this section will be on the analysis of the local approximation errors. For convenience, we assume that all the subdomains $\{\omega_{i}\}_{i=1}^{M}$ and the oversampling domains $\{\omega^{\ast}_{i}\}_{i=1}^{M}$ are resolved by the mesh.


To start with, we recall that $V_{h}\subset H^{1}(\Omega)$ is the standard FE space of continuous piecewise polynomials defined on the mesh $\tau_{h}$, and define the following local FE spaces on the oversampling domains $\omega_{i}^{\ast}$:
\begin{equation}\label{eq:4-2}
\begin{array}{lll}
{\displaystyle  {V}_{h}(\omega^{\ast}_{i}) = \big\{v_{h}|_{\omega^{\ast}_{i}}\;:\; v_{h}\in V_{h}\big\},}\\[2mm]
{\displaystyle {V}_{h,\Gamma}(\omega^{\ast}_{i}) = \big\{v_{h}\in V_{h}(\omega_{i}^{\ast}):\; v_{h} = 0 \;\;{\rm on}\;\, \partial \omega^{\ast}_{i} \cap \Gamma\big\},  }\\[2mm]
{\displaystyle V_{h,0}(\omega^{\ast}_i)= \big\{v_{h}\in {V}_{h}(\omega^{\ast}_{i}):\;v_{h} = 0 \;\;{\rm on}\;\, \partial \omega^{\ast}_{i}\big\}, }\\[2mm]
{\displaystyle V_{h,a,\varepsilon}(\omega^{\ast}_i)= \big\{u_{h}\in V_{h,\Gamma}(\omega^{\ast}_i)\;:\; {a}_{\varepsilon,\omega^{\ast}_{i}}(u_{h},v_{h}) = 0\quad \forall v_{h}\in  V_{h,0}(\omega^{\ast}_i)\big\}.}
\end{array}
\end{equation}
Here \mm{the} $V_{h,a,\varepsilon}(\omega^{\ast}_i)$ are referred to as \emph{discrete generalized harmonic spaces}. Note that $V_{h,a,\varepsilon}(\omega^{\ast}_i)$ are non-conforming approximations of $H_{a,\varepsilon}(\omega^{\ast}_i)$, i.e., $V_{h,a,\varepsilon}(\omega^{\ast}_i) \nsubseteq H_{a,\varepsilon}(\omega^{\ast}_i)$. This makes the analysis of the discrete MS-GFEM much more involved than its continuous counterpart.

In parallel with the presentation in \Cref{sec:3}, we introduce the following discrete local problem: Find $\psi_{h,\varepsilon,i}\in {V}_{h,\Gamma}(\omega^{\ast}_{i})$ such that
\begin{equation}\label{eq:4-3}
{a}_{\varepsilon,\,\omega^{\ast}_{i}}(\psi_{h,\varepsilon,i}, v_{h}) = F_{\omega_{i}^{\ast}}(v_{h})\qquad \forall v_{h}\in {V}_{h,\Gamma}(\omega^{\ast}_{i}).
\end{equation}
Similar to the continuous case, it is easy to see that $u_{h,\varepsilon}|_{\omega^{\ast}_i} -  \psi_{h,\varepsilon,i}\in V_{h,a,\varepsilon}(\omega^{\ast}_i)$, where $u_{h,\varepsilon}$ is the solution of the fine-scale FE problem \cref{eq:2-6}. To find the desired discrete local approximation space, we proceed in a similar way as before by first defining the operator $P_{h,i}: V_{h,a,\varepsilon}(\omega^{\ast}_i)\rightarrow V_{h,0}(\omega_i)$ such that
\begin{equation}\label{eq:4-4}
P_{h,i}v_{h} = I_{h}(\chi_{i}v_{h}),
\end{equation}
where $I_{h}:C(\overline{\Omega})\rightarrow V_{h}$ denotes the standard Lagrange interpolation operator, and then considering the associated Kolmogrov $n$-width:
\begin{equation}\label{eq:4-5}
d_{h,n}(\omega_{i},\omega_{i}^{\ast}):=\inf_{Q(n)\subset V_{h,0}(\omega_i)}\sup_{u_{h}\in V_{h,a,\varepsilon}(\omega^{\ast}_i)} \inf_{v_{h}\in Q(n)}\frac {\Vert P_{h,i}u_{h}-v_{h}\Vert_{a,\varepsilon,\omega_i}}{\Vert u_{h} \Vert_{a,\varepsilon,\omega_{i}^{\ast}}}.
\end{equation}

Since $V_{h,a,\varepsilon}(\omega^{\ast}_i)$ and $V_{h,0}(\omega_i)$ are finite-dimensional spaces, it is clear that $P_{h,i}$ is compact. Similar to \cref{lem:3-2}, we have the following characterization of $d_{h,n}(\omega_{i},\omega_{i}^{\ast})$. The proof is identical to that of \cref{lem:3-2} and thus is omitted here.
\begin{lemma}\label{lem:4-1}
For each $k\in \mathbb{N}$, let $(\lambda_{h,k},\,\phi_{h,k})$ be the $k$-th eigenpair (in decreasing order) of the problem
\begin{equation}\label{eq:4-6}
{a}_{\varepsilon,\omega_{i}}\big(I_{h}(\chi_{i}\phi_{h}), I_{h}(\chi_{i} v_{h})\big) = \lambda_{h}\,{a}_{\varepsilon,\omega^{\ast}_{i}}(\phi_{h}, v_{h})\quad \forall v_{h}\in V_{h,a,\varepsilon}(\omega^{\ast}_i).
\end{equation}
Then, $d_{h,n}(\omega_{i},\omega_{i}^{\ast}) = \lambda^{1/2}_{h,n+1}$, and the optimal approximation space is given by
\begin{equation}\label{eq:4-7}
\widehat{Q}(n) = {\rm span}\big\{I_{h}(\chi_{i}\phi_{h,1}),\cdots,I_{h}(\chi_{i}\phi_{h,n})\big\}.
\end{equation}
\end{lemma}

Now we can define the local particular functions and the local approximation spaces for the discrete MS-GFEM.
\begin{theorem}\label{thm:4-1}
On each $\omega_{i}$, let the discrete local particular function and the discrete local approximation space be defined as
\begin{equation}\label{eq:4-8}
u_{h,i}^{p} = \psi_{h,\varepsilon,i}|_{\omega_{i}},\quad S_{h,n_i}(\omega_i)={\rm span}\big\{\phi_{h,1}|_{\omega_i},\cdots,\phi_{h,n_i}|_{\omega_i}\big\},
\end{equation} 
where $\psi_{h,\varepsilon,i}$ is the solution of \cref{eq:4-3} and $\phi_{h,k}$ denotes the $k$-th eigenfunction of problem \cref{eq:4-6}. Then, the fine-scale FE solution $u_{h,\varepsilon}$ can be approximated locally as
\begin{equation}\label{eq:4-9}
\inf_{\varphi_{h} \in u_{h,i}^{p}+S_{h,n_{i}}(\omega_i)} \Vert I_{h}\big(\chi_{i}(u_{h,\varepsilon}-\varphi_{h})\big)\Vert_{a,\varepsilon,\omega_{i}}\leq d_{h,n_{i}}(\omega_{i},\omega_{i}^{\ast})\,\Vert u_{h,\varepsilon} - \psi_{h,\varepsilon,i} \Vert_{a, \varepsilon,\omega_{i}^{\ast}}.
\end{equation}
\end{theorem}
\begin{proof}
Similar to the continuous case, we have $u_{h,\varepsilon}|_{\omega_i^{\ast}} - \psi_{h,\varepsilon,i} \in V_{h,a,\varepsilon}(\omega^{\ast}_i)$. Hence, the estimate \cref{eq:4-9} follows from \cref{lem:4-1} and the definition of the $n$-width.
\end{proof}

\begin{rem}
Compared with those at the continuous level, the local eigenproblems at the discrete level are defined in a slightly different way, with the Lagrange interpolation operator $I_{h}$ incorporated. This modification enables us to get the desired discrete local approximation error estimates \cref{eq:4-9} involving $I_{h}$.
\end{rem}

Exponential error bounds for the local approximations of the discrete MS-GFEM are established in the next subsection. 

\subsection{Local approximation error estimates}
In this subsection, we shall derive upper bounds for the local approximation errors of the discrete MS-GFEM. It is important to note that although the local approximations of MS-GFEM in the continuous and discrete settings are constructed in the same spirit and are similar in form, the local error analysis in the discrete setting is typically more complicated. This complication becomes much greater for singularly perturbed problems, and special care is needed to deal with the interplay among the singular perturbation parameter, the fine-scale FE mesh size, and the oversampling size. In particular, we will distinguish two cases: $h\leq \varepsilon$ and $h\geq \varepsilon$. As in \cref{sec:3-2}, the subscript $i$ is dropped for ease of notation and we denote by $\delta^{\ast}={\rm dist}\,\big(\omega, \, \partial \omega^{\ast}\setminus\partial \Omega\big)$. 

To begin with, we state the main results of this subsection. Let us first consider the case when only the local particular functions are used for the local approximations.
\begin{theorem}\label{thm:4-2}
Let $u_{h,\varepsilon}$ be the solution of the fine-scale FE problem \cref{eq:2-6}, and let $\psi_{h,\varepsilon}$ be the local particular function defined in \cref{eq:4-3}. There exist positive constants $c_{0}$, $c_{1}$, and $C$ independent of $\varepsilon$ and $h$, such that if $\delta^{\ast}\geq c_{0} \max\{\varepsilon,h\}$,
\begin{equation}\label{eq:4-12-0}
\begin{array}{lll}
{\displaystyle \big\Vert I_{h}\big(\chi (u_{h,\varepsilon} - \psi_{h,\varepsilon})\big)\big\Vert_{a,\varepsilon,\omega}\,\big/ \,\big(\Vert u_{h,\varepsilon} \Vert_{L^{2}(\omega^{\ast})} + \Vert \psi_{h,\varepsilon} \Vert_{L^{2}(\omega^{\ast})}\big) }\\[3mm]
{\displaystyle \leq 
\begin{cases}
C\big(1+\varepsilon \Vert\nabla \chi\Vert_{L^{\infty}(\omega)}\big) (\varepsilon/\delta^{\ast})\,e^{-c_{1}\delta^{\ast}/\varepsilon},\;\quad\quad {\rm if}\;\, h\leq \varepsilon,\\[2mm]

C\big(1+\varepsilon \Vert\nabla \chi\Vert_{L^{\infty}(\omega)}\big) (h/\delta^{\ast})^{1/2}\,e^{-c_{1}\delta^{\ast}/h},\quad {\rm if}\;\, \varepsilon\leq h.
\end{cases}
}
\end{array}
\end{equation}
\end{theorem}

\Cref{thm:4-2} indicates that if the oversampling size $\delta^{\ast}$ is relatively large with respect to $\varepsilon$ and $h$, the local particular functions are good local approximations of the fine-scale FE solution and thus it is not necessary to build the local approximation spaces. Conversely, if relatively small oversampling domains are used, then we need to construct the local approximation spaces using the local eigenfunctions and identify the associated approximation errors. To do this, as before, we assume that $\delta^{\ast}>0$.
\begin{theorem}\label{thm:4-3}
There exist positive constants $\Lambda$, $b$, $c_{0}$, $c_{1}$, and $C$, independent of $\varepsilon$ and $h$, such that the following holds.
\begin{itemize}
\item $h\leq \varepsilon$. For any $n>\Lambda$, if $h\leq c_{0}n^{-1/(d+1)}$, then 
\begin{equation}\label{eq:4-13}
d_{n}(\omega,\omega^{\ast}) \leq C\big(1+\varepsilon \Vert\nabla \chi\Vert_{L^{\infty}(\omega)}\big)\varepsilon e^{-bn^{{1}/{(d+1)}}}.
\end{equation}
\item $h\geq \varepsilon$. Let $\sigma_{0}>\sqrt{2}$ and assume that the meshes $\{\tau_{h}\}$ are quasi-uniform. For any $n$ with $c_{1}\big(\delta^{\ast}/(\sigma_{0}^{2}h)\big)^{d+1}\leq n\leq c_{1}\big(\delta^{\ast}/(2h)\big)^{d+1}$, 
\begin{equation}\label{eq:4-14}
d_{n}(\omega,\omega^{\ast}) \leq C\big(1+\varepsilon \Vert\nabla \chi\Vert_{L^{\infty}(\omega)}\big)\,e^{-b\sigma_{0}^{-d/(d+1)}n^{{1}/{(d+1)}}}.
\end{equation}
\end{itemize}
\end{theorem}
\begin{rem}
In the asymptotic regime (i.e., $h\leq \varepsilon$), the local error bounds of the discrete method in \cref{thm:4-2,thm:4-3} are very similar to those of the continuous method in \cref{thm:3-0,thm:3-2}, respectively. In the preasymptotic regime (i.e., $h\geq \varepsilon$), which is of more interest in practice, the results are different and more complicated. In particular, the performance of the discrete method in this regime depends crucially on the ratio of the oversampling size $\delta^{\ast}$ to the mesh size $h$ instead of to $\varepsilon$. For example, in this regime, even if $\delta^{\ast}/\varepsilon$ is large, the (discrete) local particular functions may not locally approximate the discrete solution well, as contrasted with the continuous case.
\end{rem}

To prove \cref{thm:4-2,thm:4-3}, we need some Caccioppoli-type inequalities for discrete generalized harmonic functions as in the continuous case. These inequality, however, are much more difficult to prove than their continuous counterparts due to the spatial discretization. To prove such inequalities, we first give a preliminary superapproximation result.

\begin{lemma}[\cite{demlow2011local}]\label{lem:4-2}
Let $\eta\in C^{\infty}(\overline{\Omega})$ with $|\eta|_{W^{j,\infty}(\Omega)}\leq C\delta^{-j}$ for $0\leq j\leq r+1$. Then for each $u_{h}\in V_{h}$ and $K\in \tau_{h}$ satisfying $h_{K} \leq \delta$, 
\begin{align}
{\displaystyle \Vert \eta^{2}u_{h}- I_{h}(\eta^{2}u_{h})\Vert_{H^{1}(K)}\leq C\big(\frac{h_{K}}{\delta}\Vert \nabla (\eta u_{h})\Vert_{L^{2}(K)}+\frac{h_{K}}{\delta^{2}}\Vert u_{h}\Vert_{L^{2}(K)}\big),\label{eq:4-15}}\\
{\displaystyle \Vert \eta^{2}u_{h}- I_{h}(\eta^{2}u_{h})\Vert_{L^{2}(K)}\leq C\big(\frac{h^{2}_{K}}{\delta}\Vert \nabla(\eta u_{h})\Vert_{L^{2}(K)}+\frac{h^{2}_{K}}{\delta^{2}}\Vert u_{h}\Vert_{L^{2}(K)}\big),\label{eq:4-16}}
\end{align}
where $C$ is independent of $\delta$ and $h_{K}$.
\end{lemma}

The following lemma gives the desired discrete Caccioppoli inequalities in both asymptotic and preasymptotic regimes.
\begin{lemma}\label{lem:4-3}
Let $D\subset D^{\ast}$ be given open connected subsets of $\Omega$, and let $\delta:={\rm dist}\,\big(D, \, \partial D^{\ast}\setminus\partial \Omega\big)>0$. In addition, let $\max_{K\cap D^{\ast}\neq \emptyset}h_{K}\leq \frac{1}{2}\delta$. Then, there exists $C_{0}>0$ independent of $\varepsilon$, $h$, and $\delta$, such that for any $u_{h}\in V_{h,a,\varepsilon}(D^{\ast})$,
\begin{equation}\label{eq:4-17}
\begin{array}{lll}
{\displaystyle \Vert u_{h}\Vert_{a,\varepsilon,D}\leq C_{0}(\varepsilon/\delta)\Vert u_{h} \Vert_{L^{2}(D^{\ast})}, \qquad {\rm if}\;\, h\leq \varepsilon,}\\[3mm]
{\displaystyle \Vert u_{h}\Vert_{a,\varepsilon,D}\leq C_{0}(h/\delta)^{1/2}\Vert u_{h} \Vert_{L^{2}(D^{\ast})}, \quad {\rm if}\;\, \varepsilon\leq h.}
\end{array}
\end{equation}
\end{lemma}
\begin{proof}
Let $\widetilde{D}^{\ast}$ be the union of elements that are contained in $D^{\ast}$. Using the assumption that $\delta={\rm dist}\,\big(D, \, \partial D^{\ast}\setminus\partial \Omega\big)>2\max_{K\cap D^{\ast}\neq \emptyset}h_{K}$, we see that 
\begin{equation}\label{eq:4-18}
{\rm dist}\,\big(D, \, \partial \widetilde{D}^{\ast}\setminus\partial \Omega\big) \geq \frac{1}{2}\delta.
\end{equation}
Let $\eta\in C^{\infty}(\widetilde{D}^{\ast})$ be a cut-off function which satisfies 
\begin{equation}\label{eq:4-19}
\begin{array}{lll}
{\displaystyle \eta \equiv 1\;\;\; {\rm in} \;\;\,D,\quad \eta = 0\;\; \;{\rm on}\;\;\,\partial \widetilde{D}^{\ast}\setminus\partial \Omega,}\\[2mm] 
{\displaystyle |\eta|_{W^{j,\infty}(\widetilde{D}^{\ast})}\leq C\delta^{-j},\;\;\;j=1,\cdots, r+1. }
\end{array}
\end{equation}
Using the same argument as in the proof of \cref{eq:3-8}, we have
\begin{equation}\label{eq:4-20}
\Vert \eta u_{h}\Vert^{2}_{a,\varepsilon, \widetilde{D}^{\ast}} = \varepsilon^{2}\int_{\widetilde{D}^{\ast}}(A\nabla \eta \cdot \nabla \eta) u_{h}^{2}\,d{\bm x} +{a}_{\varepsilon,\widetilde{D}^{\ast}}(u_{h},\eta^{2}u_{h}).
\end{equation}
In contrast to the continuous case, the last term on the right hand side of \cref{eq:4-20} does not vanish since $\eta^{2}u_{h} \notin V_{h,0}(\widetilde{D}^{\ast})$. However, noting that $I_{h}(\eta^{2}u_{h}) \in V_{h,0}(\widetilde{D}^{\ast})$, where $I_{h}$ denotes the standard Lagrange interpolant, we see that ${a}_{\varepsilon,\widetilde{D}^{\ast}}(u_{h},I_{h}(\eta^{2}u_{h})) = 0$. It follows that
\begin{equation}\label{eq:4-21}
\begin{array}{lll}
{\displaystyle \Vert \eta u_{h}\Vert^{2}_{a,\varepsilon, \widetilde{D}^{\ast}} = \varepsilon^{2}\int_{\widetilde{D}^{\ast}}(A\nabla \eta \cdot \nabla \eta) u_{h}^{2}\,d{\bm x} +{a}_{\varepsilon,\widetilde{D}^{\ast}}\big(u_{h},\eta^{2}u_{h} - I_{h}(\eta^{2}u_{h})\big) }\\[4mm]
{\displaystyle\qquad   \leq C(\varepsilon/\delta)^{2} \Vert u_{h} \Vert^{2}_{L^{2}(\widetilde{D}^{\ast})} + \varepsilon^{2} \big(A\nabla u_{h}, \nabla (\eta^{2}u_{h} - I_{h}(\eta^{2}u_{h}))\big)_{L^{2}(\widetilde{D}^{\ast})} }\\[3mm]
{\displaystyle \qquad \quad + \,\big(u_{h}, \eta^{2}u_{h} - I_{h}(\eta^{2}u_{h})\big)_{L^{2}(\widetilde{D}^{\ast})}.}
\end{array}
\end{equation}

To estimate the last two terms of \cref{eq:4-21}, we shall use the superapproximation result in \cref{lem:4-2}. Applying \cref{eq:4-15} and an inverse estimate (local to each element), we have
\begin{equation}\label{eq:4-22}
\begin{array}{lll}
{\displaystyle \mm{\big|} \big(A\nabla u_{h}, \nabla (\eta^{2}u_{h} - I_{h}(\eta^{2}u_{h}))\big)_{L^{2}(\widetilde{D}^{\ast})} \mm{\big|}}\\[2mm]
{\displaystyle \leq C\sum_{K\subset \widetilde{D}^{\ast}}h_{K}\Vert \nabla u_{h}\Vert_{L^{2}(K)} \big(\delta^{-1} \Vert\nabla (\eta u_{h})\Vert_{L^{2}(K)} + \delta^{-2} \Vert u_{h}\Vert_{L^{2}(K)}\big)}\\[4mm]
{\displaystyle \leq C\delta^{-2}\sum_{K\subset \widetilde{D}^{\ast}}  \Vert u_{h}\Vert^{2}_{L^{2}(K)} + \frac{a_{\rm min}}{3} \sum_{K \subset \widetilde{D}^{\ast}} \Vert\nabla (\eta u_{h})\Vert^{2}_{L^{2}(K)}}\\[3mm]
{\displaystyle \leq C\delta^{-2}\Vert u_{h}\Vert^{2}_{L^{2}(\widetilde{D}^{\ast})} + \frac{1}{3}\Vert A^{1/2}\nabla(\eta u_{h})\Vert^{2}_{L^{2}(\widetilde{D}^{\ast})}.}
\end{array}
\end{equation}
If $h\leq \varepsilon$, we can use \cref{eq:4-16} and a similar argument as above to deduce
\begin{equation}\label{eq:4-23}
\begin{array}{lll}
{\displaystyle \mm{\big|}\big(u_{h}, \eta^{2}u_{h} - I_{h}(\eta^{2}u_{h})\big)_{L^{2}(\widetilde{D}^{\ast})}\mm{\big|}}\\[3mm]
{\displaystyle  \leq \,C(h/\delta)^{2}\Vert u_{h}\Vert^{2}_{L^{2}(\widetilde{D}^{\ast})} + \frac{h^{2}}{3}\Vert A^{1/2}\nabla(\eta u_{h})\Vert^{2}_{L^{2}(\widetilde{D}^{\ast})} }\\[3mm]
{\displaystyle \leq\,C(\varepsilon/\delta)^{2}\Vert u_{h}\Vert^{2}_{L^{2}(\widetilde{D}^{\ast})} + \frac{\varepsilon^{2}}{3}\Vert A^{1/2}\nabla(\eta u_{h})\Vert^{2}_{L^{2}(\widetilde{D}^{\ast})}.}
\end{array}
\end{equation}
Inserting \cref{eq:4-22,eq:4-23} into \cref{eq:4-21} and using \cref{eq:4-19} and the fact that $\widetilde{D}^{\ast} \subset D^{\ast}$ gives \cref{eq:4-17} in the case that $h\leq \varepsilon$. 

Next we assume that $\varepsilon\leq h$. Using \cref{eq:4-16}, \cref{eq:4-19}, and an inverse estimate, we see that for each $K\subset \widetilde{D}^{\ast}$, 
\begin{equation}\label{eq:4-24}
\Vert \eta^{2}u_{h}- I_{h}(\eta^{2}u_{h})\Vert_{L^{2}(K)}\leq C\Big(\frac{h^{2}_{K}}{\delta^{2}}+\frac{h_{K}}{\delta} \Big)\Vert u_{h}\Vert_{L^{2}(K)}\leq C\frac{h_{K}}{\delta} \Vert u_{h} \Vert_{L^{2}(K)}.
\end{equation}
It follows that
\begin{equation}\label{eq:4-25}
\mm{\big|}\big(u_{h}, \eta^{2}u_{h} - I_{h}(\eta^{2}u_{h})\big)_{L^{2}(\widetilde{D}^{\ast})} \mm{\big|} \leq C(h/\delta) \Vert u_{h}\Vert^{2}_{L^{2}(\widetilde{D}^{\ast})}.
\end{equation}
Inserting \cref{eq:4-22,eq:4-25} into \cref{eq:4-21} and using the assumption that $\varepsilon\leq h$ gives the second inequality of \cref{eq:4-17}.
\end{proof}

If the distance between $D$ and $D^{\ast}$ is large with respect to $\varepsilon$ and $h$, we can get a sharper Caccioppoli inequality for discrete generalized harmonic functions as in the continuous case; see \cref{lem:3-4}.
\begin{lemma}\label{lem:4-4}
There exist positive constants $c_{0}$, $c_{1}$ and $C$ such that for any open connected sets $D\subset D^{\ast}\subset \Omega$ with $\delta:={\rm dist}\,\big(D, \, \partial D^{\ast}\setminus\partial \Omega\big) \geq c_{0} \max\{\varepsilon,h\}$ and for any $u_{h}\in V_{h,a,\varepsilon}(D^{\ast})$,
\begin{equation}\label{eq:4-26}
\begin{array}{lll}
{\displaystyle \Vert u_{h}\Vert_{a,\varepsilon,D}\leq C(\varepsilon/\delta)e^{-c_{1}\delta/\varepsilon}\Vert u_{h} \Vert_{L^{2}(D^{\ast})}, \;\qquad {\rm if}\;\, h\leq \varepsilon,}\\[3mm]
{\displaystyle \Vert u_{h}\Vert_{a,\varepsilon,D}\leq C(h/\delta)^{1/2} e^{-c_{1}\delta/h} \Vert u_{h} \Vert_{L^{2}(D^{\ast})}, \quad {\rm if}\;\, \varepsilon\leq h.}
\end{array}
\end{equation}
\end{lemma}
\begin{proof}
Estimate \cref{eq:4-26} can be proved by an iteration argument similar to that in the proof of \cref{lem:3-4} and thus we only prove it for the case $\varepsilon\leq h$. \mm{Let $C_{d}$ and $C_{0}$ be the constants as in \cref{eq:3-0,lem:4-3}, respectively,} and we assume that $C^{2}_{0}\mm{C_{d}}e^{2}h<\delta$. Let $N = \lfloor \delta/(C_{0}^{2}\mm{C_{d}}e^{2}h)\rfloor$.
Choose $\{{D}_{k}\}_{k=0}^{N}$ such that ${D} = {D}_{0}\subset{D}_{1}\subset\cdots\subset{D}_{N} = {D}^{\ast}$ and ${\rm dist}({D}_{k-1}, \partial {D}_{k}\setminus\partial \Omega) = \delta/N$. Applying \cref{eq:4-17} in the case $\varepsilon\leq h$ on ${D}_{0}$ and ${D}_{1}$ gives
\begin{equation}\label{eq:4-27}
\Vert u_{h} \Vert_{L^{2}({D}_{0})} \leq \big(C^{2}_{0}\mm{C_{d}}Nh/\delta\big)^{1/2} \,\Vert u_{h} \Vert_{L^{2}({D}^{1})}.
\end{equation}
Repeating the argument on domains ${D}_{1},\cdots,{D}_{N}$, it follows that
\begin{equation}\label{eq:4-28}
\Vert u_{h} \Vert_{L^{2}({D})}=\Vert u_{h} \Vert_{L^{2}({D}_{0})} \leq \big(C^{2}_{0}\mm{C_{d}}Nh/\delta\big)^{N/2} \,\Vert u_{h} \Vert_{L^{2}({D}^{\ast})}.
\end{equation}
Noting that $N$ satisfies $C^{2}_{0}\mm{C_{d}}Nh/\delta\leq e^{-2}$ and $N\geq \delta/(C_{0}^{2}\mm{C_{d}}e^{2}h)-1$, we further have
\begin{equation}\label{eq:4-29}
\Vert u_{h} \Vert_{L^{2}({D})}\leq e^{-N} \;\Vert u_{h} \Vert_{L^{2}({D}^{\ast})} \leq e^{1-\delta/(C_{0}^{2}\mm{C_{d}}e^{2}h)}\,\Vert u_{h} \Vert_{L^{2}({D}^{\ast})}.
\end{equation}
To finish the proof, let ${D}_{\delta/2}$ lie between ${D}$ and ${D}^{\ast}$ with ${\rm dist}({D}, \partial {D}_{\delta/2}\setminus\partial \Omega) = {\rm dist}({D}_{\delta/2}, \partial {D}^{\ast}\setminus\partial \Omega)=\delta/2$. The desired estimate \cref{eq:4-26} follows by first applying \cref{eq:4-17} on $D$ and $D_{\delta/2}$ and then applying \cref{eq:4-29} on $D_{\delta/2}$ and $D^{\ast}$.
\end{proof}

The following lemma gives the stability of the operator $P_{h}$ defined by \cref{eq:4-4}.
\begin{lemma}\label{lem:4-5}
Let $D^{\ast}\subset \Omega$ be a collection of elements and $D\subset D^{\ast}$. Assume that $\eta \in W^{1,\infty}(D^{\ast})$ satisfies $\Vert\eta\Vert_{L^{\infty}(D^{\ast})}\leq 1$ and ${\rm supp}(\eta)\subset \overline{D}$. Then, there exists $C>0$ such that 
\begin{equation}\label{eq:4-25-0}
 \Vert I_{h}(\eta u_{h})\Vert_{a,\varepsilon,D^{\ast}}\leq C\big(1+\varepsilon\Vert \nabla \eta\Vert_{L^{\infty}(D^{\ast})}\big)\Vert u_{h}\Vert_{a,\varepsilon,D}, \quad \forall u_{h}\in V_{h}(D^{\ast}).
\end{equation}
\end{lemma}
\begin{proof}
Using error estimates of the interpolation operator $I_{h}$ (see, e.g., \cite[Theorem 4.4.20]{brenner2008mathematical}) and a triangle inequality gives that
\begin{equation}
\Vert I_{h}(\eta u_{h})\Vert_{a,\varepsilon,D^{\ast}} \leq C\Vert \eta u_{h}\Vert_{a,\varepsilon,D^{\ast}}.    
\end{equation}
Since $\Vert\eta\Vert_{L^{\infty}(D^{\ast})}\leq 1$ and ${\rm supp}(\eta)\subset \overline{D}$, we further have
\begin{equation}\label{eq:4-25-1}
\Vert I_{h}(\eta u_{h})\Vert_{a,\varepsilon,D^{\ast}}\leq C\big(\varepsilon \Vert A^{1/2}\nabla(\eta u_{h})\Vert_{L^{2}(D^{\ast})} + \Vert u_{h}\Vert_{L^{2}(D)}\big).
\end{equation}
Applying a triangle inequality and the assumptions on $\eta$ again, we see that
\begin{equation}\label{eq:4-25-2}
\begin{array}{lll}
{\displaystyle \Vert A^{1/2}\nabla(\eta u_{h})\Vert_{L^{2}(D^{\ast})}\leq \Vert \eta A^{1/2}\nabla u_{h}\Vert_{L^{2}(D^{\ast})} + \Vert u_{h} A^{1/2}\nabla \eta \Vert_{L^{2}(D^{\ast})} }\\[2mm]
{\displaystyle \qquad \leq \Vert A^{1/2}\nabla u_{h}\Vert_{L^{2}(D)} + C\Vert \nabla \eta\Vert_{L^{\infty}(D^{\ast})} \Vert u_{h} \Vert_{L^{2}(D)}.}
\end{array}
\end{equation}
Inserting \cref{eq:4-25-2} into \cref{eq:4-25-1} yields \cref{eq:4-25-0}.
\end{proof}

Now we are ready to prove \cref{thm:4-2}.
\begin{proof}[Proof of \cref{thm:4-2}]
Assuming that $h\leq \varepsilon$ and applying \cref{lem:4-4,lem:4-5} with $D=\omega$, $D^{\ast} = \omega^{\ast}$, $u_{h}=u_{h,\varepsilon}|_{\omega^{\ast}}-\psi_{h,\varepsilon}\in V_{h,a,\varepsilon}(\omega^{\ast})$, and $\eta = \chi$, we get
\begin{equation}\label{eq:4-25-3}
\begin{array}{lll}
{\displaystyle \big\Vert I_{h}\big(\chi (u_{h,\varepsilon} - \psi_{h,\varepsilon})\big)\big\Vert_{a,\varepsilon,\omega}\leq C\big(1+\varepsilon \Vert\nabla \chi\Vert_{L^{\infty}(\omega)}\big) (\varepsilon/\delta^{\ast})e^{-c_{1}\delta^{\ast}/\varepsilon}}\\[3mm]
{\displaystyle \qquad \qquad \qquad \qquad \qquad \qquad \qquad \times\,\Vert u_{h,\varepsilon}-\psi_{h,\varepsilon} \Vert_{L^{2}(\omega^{\ast})}.}
\end{array}
\end{equation}
Estimate \cref{eq:4-12-0} follows from \cref{eq:4-25-3} and the fact that $\Vert \psi_{h,\varepsilon}\Vert_{L^{2}(\omega^{\ast})} \leq \Vert f\Vert_{L^{2}(\omega^{\ast})}$. The other case can be proved similarly.
\end{proof}


It remains to prove \cref{thm:4-3}. To this end, we first construct an auxiliary space in which any discrete generalized harmonic function can be approximated with an algebraic convergence rate as in the continuous case; see \cref{lem:3-5}.
\begin{lemma}\label{lem:4-6}
Let $D$ and $D^{\ast}$ be open connected subsets of $\Omega$ with $D\subset D^{\ast}$ and $\delta={\rm dist}\,\big(D, \, \partial D^{\ast}\setminus\partial \Omega\big)>4h$, and let $m\in \mathbb{N}$ satisfy $m\geq C_{1}|D^{\ast}|(\delta/2)^{-d}$ with $C_{1}$ given by \cref{lem:3-3-0}. There exist an $m$-dimensional space $\Psi_{h,m}(D^{\ast})\subset V_{h,a,\varepsilon}(D^{\ast})$ and a constant $C>0$ independent of $\varepsilon$, $h$, and $\delta$ such that the following holds.
\begin{itemize}
\item[(i)] Supposing that $h\leq \varepsilon$, then for any $u_{h}\in V_{h,a,\varepsilon}(D^{\ast})$,
\begin{equation}\label{eq:4-30}
\inf_{\varphi_{h}\in \Psi_{h,m}({D}^{\ast})} \Vert u_{h}-\varphi_{h}\Vert_{a,\varepsilon,{D}}\leq C |{D}^{\ast}|^{1/d}m^{-1/d}\delta^{-1}\,\Vert u_{h} \Vert_{a,\varepsilon,{D}^{\ast}}.
\end{equation}
\item[(ii)] Supposing that $h\geq\varepsilon$ and that the meshes $\{ \tau_{h}\}$ are quasi-uniform, then for any $u_{h}\in V_{h,a,\varepsilon}(D^{\ast})$,
\begin{equation}\label{eq:4-31}
\inf_{\varphi_{h}\in {\Psi}_{h,m}({D}^{\ast})} \Vert u_{h}-\varphi_{h}\Vert_{a,\varepsilon,{D}}\leq C |{D}^{\ast}|^{1/d}m^{-1/d}(h\delta)^{-1/2}\,\Vert u_{h} \Vert_{a,\varepsilon,{D}^{\ast}}.
\end{equation}
\end{itemize}
\end{lemma}
\begin{proof}
The proof of the result (i) is exactly the same as that of \cref{lem:3-5}, based on a combination of \cref{lem:3-3} and the discrete Caccioppoli inequality \cref{eq:4-17}. Now we assume that $h\geq \varepsilon$. Let $D_{\delta/2}$ satisfy ${D}\subset D_{\delta/2}\subset{D}^{\ast}$ and ${\rm dist}({D},\,\partial {D}_{\delta/2}\setminus\partial \Omega) = {\rm dist}({D}_{\delta/2},\,\partial {D}^{\ast}\setminus\partial \Omega) = \delta/2$. Applying \cref{lem:3-3} on $D_{\delta/2}$ and $D^{\ast}$ with $\mathcal{S}(D^{\ast}) = V_{h,a,\varepsilon}(D^{\ast})$, we see that there exists an $m$-dimensional space ${\Psi}_{h,m}(D^{\ast})\subset V_{h,a,\varepsilon}(D^{\ast})$ such that for any $u_{h}\in V_{h,a,\varepsilon}(D^{\ast})$,
\begin{equation}\label{eq:4-35}
\inf_{\varphi_{h}\in {\Psi}_{h,m}({D}^{\ast})} \Vert u_{h}-\varphi_{h}\Vert_{L^{2}(D_{\delta/2})}\leq C|{D}^{\ast}|^{1/d}m^{-1/d} \Vert \nabla u_{h} \Vert_{L^{2}({D}^{\ast})}.
\end{equation}
Here we can slightly shrink the domain $D^{\ast}$ in \cref{eq:4-35} if necessary such that it is made of elements. Using an inverse estimate on $D^{\ast}$, we further have
\begin{equation}\label{eq:4-36}
\begin{array}{lll}
{\displaystyle \inf_{\varphi_{h}\in {\Psi}_{h,m}({D}^{\ast})} \Vert u_{h}-\varphi_{h}\Vert_{L^{2}(D_{\delta/2})}\leq Ch^{-1}|{D}^{\ast}|^{1/d}m^{-1/d} \Vert u_{h} \Vert_{L^{2}({D}^{\ast})} }\\[2mm]
{\displaystyle \qquad \qquad \quad \qquad \qquad \qquad \qquad \leq Ch^{-1}|{D}^{\ast}|^{1/d}m^{-1/d}\Vert u_{h} \Vert_{a,\varepsilon,{D}^{\ast}}.}
\end{array}
\end{equation}
Combining \cref{eq:4-36} with the second part of \cref{eq:4-17} yields \cref{eq:4-31}.
\end{proof}

Now we are in the position to prove \cref{thm:4-3}.
\begin{proof}[Proof of \cref{thm:4-3}]
Let $N\in \mathbb{N}$. As in the continuous case, we choose nested domains $\{\omega^{j}\}_{j=1}^{N+1}$ with $\omega=\omega^{N+1}\subset \omega^{N}\subset\cdots\subset\omega^{1} = \omega^{\ast}$ and ${\rm dist}(\omega^{k-1},\partial \omega^{k}\setminus\partial \Omega) = \delta^{\ast}/N$. The proof of the first part of \cref{thm:4-3} is similar to that of \cref{thm:3-2}. By first repeatedly applying \cref{eq:4-30} on $\omega^{1},\cdots,\omega^{N}$, then using \cref{lem:4-5}, and finally taking $n\sim N^{d+1}$, we get \cref{eq:4-13} provided that $h\leq \delta^{\ast}/(4N)\leq c_{0}n^{-1/(d+1)}$.

Next we prove the second part of \cref{thm:3-2}. Let $\sigma_{0}>\sqrt{2}$. Assuming that 
\begin{equation}\label{eq:4-37}
\frac{\delta^{\ast}}{\sigma_{0}^{2}h}\leq N\leq \frac{\delta^{\ast}}{2h}\;\; \Longleftrightarrow \;\;\frac{\delta^{\ast}}{N}>2h,\; \Big(\frac{\delta^{\ast}}{Nh}\Big)^{1/2}\leq \sigma_{0},
\end{equation}
we can apply \cref{lem:4-6} (ii) with $D=\omega^{k-1}$, $D^{\ast} = \omega^{k}$, and $\delta=\delta^{\ast}/N$ to deduce that there exists ${\Psi}_{h,m}({\omega}^{k})\subset V_{h,a,\varepsilon}(\omega^{k})$ such that
\begin{equation}\label{eq:4-38}
\begin{array}{lll}
{\displaystyle \inf_{\varphi_{h}\in {\Psi}_{h,m}({\omega}^{k})} \Vert u_{h}-\varphi_{h}\Vert_{a,\varepsilon,{\omega}^{k-1}}\leq C |{\omega}^{k}|^{1/d}m^{-1/d}\delta^{-1}(\delta/h)^{1/2}\,\Vert u_{h} \Vert_{a,\varepsilon,{\omega}^{k}} }\\[3mm]
{\displaystyle \qquad \qquad \leq C\sigma_{0} |{\omega}^{k}|^{1/d}m^{-1/d}\delta^{-1}\,\Vert u_{h} \Vert_{a,\varepsilon,{\omega}^{k}}\qquad \quad \forall u_{h}\in V_{h,a,\varepsilon}(\omega^{k}).}
\end{array}
\end{equation}
Note that estimate \cref{eq:4-38} is similar to \cref{eq:4-30} except that the constant $C$ is replaced by $C\sigma_{0}$. Hence, \cref{eq:4-14} can also be proved by a similar argument as in the proof of \cref{thm:3-2} and by keeping track of the constant $\sigma_{0}$. We omit the details here. Finally, since we take $n\sim N^{d+1}$ in the proof, condition \cref{eq:4-37} is equivalent to that
\begin{equation}
c_{1}\Big(\frac{\delta^{\ast}}{\sigma_{0}^{2}h}\Big)^{(d+1)}\leq n\leq c_{1}\Big(\frac{\delta^{\ast}}{2h}\Big)^{(d+1)}.
\end{equation}
\end{proof}

The global error estimates of the discrete MS-GFEM are identical to those in the continuous setting, and hence are omitted. 

\section{Numerical experiments}\label{sec:5}
In this section, we present some numerical results to support our theoretical analysis. We consider the problem \cref{eq:2-1} on the unit square, i.e., $\Omega=(0,1)^{2}$, with a scalar diffusion coefficient varying at a scale of 0.01 as illustrated in \cref{fig:5-1}. The right-hand side $f$ is given by
\begin{equation*}
    f({\bm x}) = 10\exp\big(-10(x_{1}-0.15)^{2}-10(x_{2}-0.55)^{2}\big).
\end{equation*}

\begin{figure}\label{fig:5-1}
\centering
\includegraphics[scale=0.3]{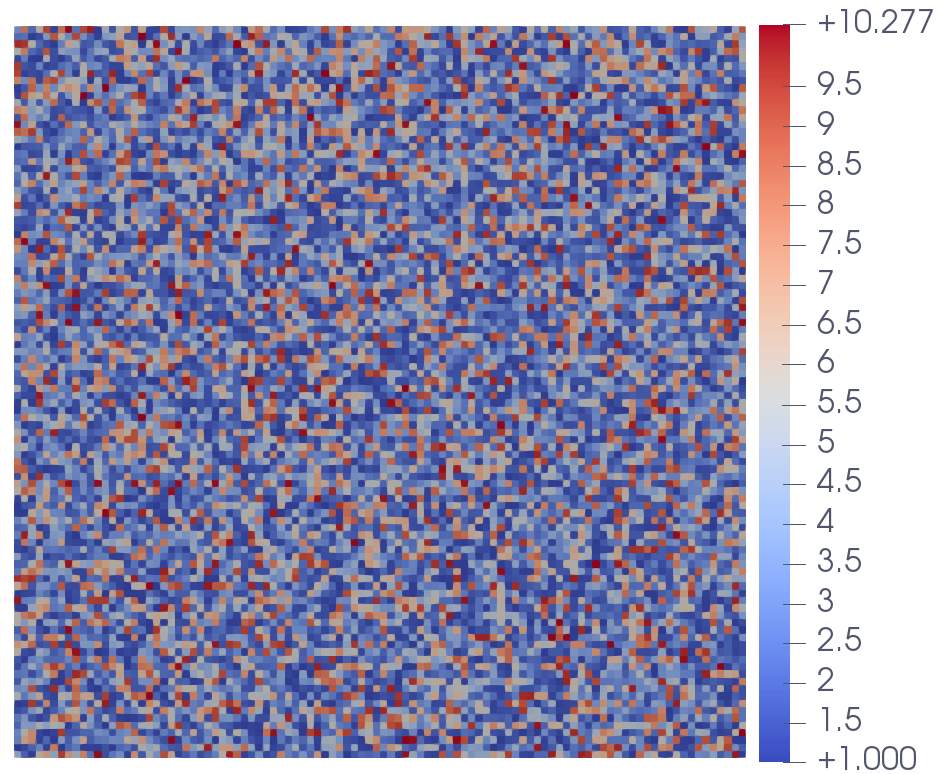}
\caption{The diffusion coefficient $A({\bm x})$.}
\end{figure}

The domain is discretized by a uniform Cartesian grid with $h=10^{-3}$ on which all the local problems in the proposed method are solved. We first partition the domain into $N^{2}$ non-overlapping subdomains $\{\omega_{i}^{\prime}\} $ with boundaries aligned with the fine mesh, and then extend each $\omega_{i}^{\prime}$ by adding 2 layers of fine-mesh elements to create an open cover $\{ \omega_{i}\}$ of the domain. Each subdomain $\omega_{i}$ is further extended by adding $\ell$ layers of fine-mesh elements to create the associated oversampling domain $\omega_{i}^{\ast}$, and hence the oversampling size $\delta^{\ast}$ is $\ell h$. The fine-scale FE solution $u_{h}$, i.e., the solution of the standard FE discretization \cref{eq:2-6} on the fine-mesh, is considered as the reference solution. 


\begin{figure}\label{fig:5-2}
\centering
\includegraphics[scale=0.3]{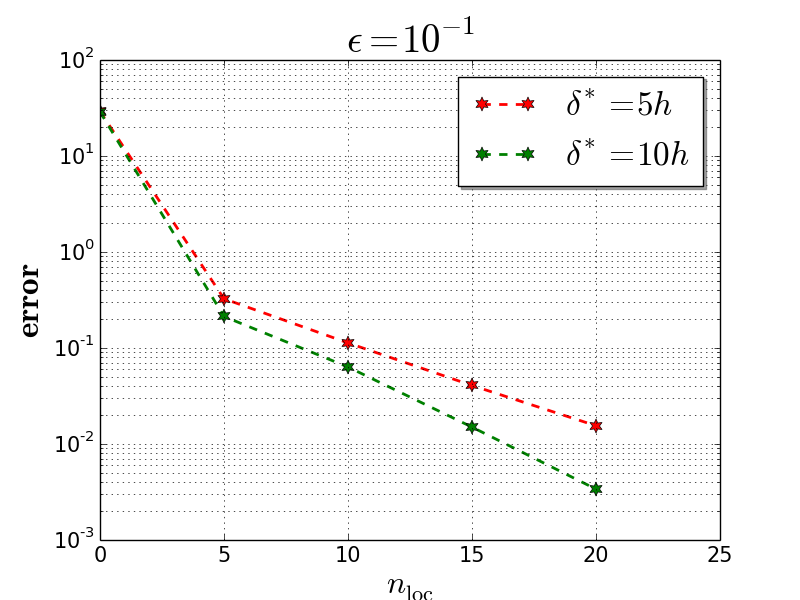}~
\includegraphics[scale=0.3]{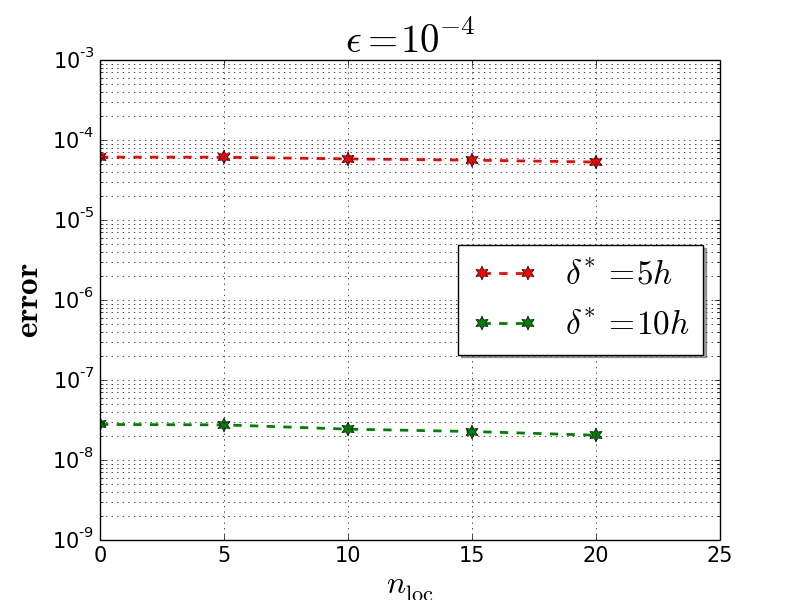}
\caption{Plots of $\Vert u_{h}-u_{h}^{G}\Vert_{a,\varepsilon}$ against $n_{\rm loc}$ (the number of local basis functions used per subdomain) for $\varepsilon=10^{-1}$ and $\varepsilon=10^{-4}$.}

\vspace{-2ex}
\end{figure}

First, we test the performance of our method using the local spectral basis for the problem with $\varepsilon=10^{-1}$ and $\varepsilon=10^{-4}$. The errors measured in the energy norm are plotted in \cref{fig:5-2} in a semi-logarithmic scale with $N=20$. We clearly see that the method behaves dramatically differently for non-singularly perturbed and singularly perturbed problems. For the problem with $\varepsilon=10^{-1}$, it is essential to use the spectral basis for the local approximations, and the errors decay nearly exponentially with respect to the number of local basis functions used per subdomain. For the problem with $\varepsilon=10^{-4}$, however, the global particular function alone can approximate the reference solution very well, and the local spectral basis does little to improve the accuracy of the method.

Next, we test the performance of our method without using the local spectral basis, i.e., by only pasting the local particular functions together to get the approximate solution. Numerical results with $N=10$ for various values of $\varepsilon$ and $\delta^{\ast}$ are shown in \cref{table:covergence history}. We observe that for $\varepsilon=10^{-1}$ and $\varepsilon=10^{-2}$, the errors decay slowly with increasing oversampling size, but for smaller $\varepsilon$, the errors decay rapidly. For a fixed oversampling size, at first the errors decay rapidly as $\varepsilon$ decreases, but they stop decaying when $\varepsilon$ is much smaller than $h$. This behavior agrees with our theoretical prediction; recall that \cref{thm:4-2} shows that the error of the method decays exponentially with respect to $\delta^{\ast}/h$ instead of $\delta^{\ast}/\varepsilon$ in the pre-asymptotic regime. We remark that the rates of convergence with respect to $\varepsilon$ shown by \cref{table:covergence history} are smaller than that predicted by \cref{thm:4-2}, which is due to the fact that the oversampling size used here is not sufficiently large to meet the assumption therein. In fact, for singularly perturbed problems, the method with a moderate oversampling size can achieve an accuracy close to machine precision.

In \cref{fig:5-3}, we plot the reference solution and the global particular solution computed with $N=10$ and $\delta^{\ast}=15h$ for the problem with $\varepsilon = 0.1$ and $\varepsilon = 10^{-3}$. It can be clearly seen that for $\varepsilon = 0.1$, the global particular solution fails to approximate the reference solution well, while for $\varepsilon = 10^{-3}$, it agrees very well with the reference solution.

\setlength\tabcolsep{10pt}
\renewcommand{\arraystretch}{1.5}
\begin{table}
\caption{$\Vert u_{h}-u_{h}^{G}\Vert_{a,\varepsilon}$.}\label{table:covergence history}
\centering
\begin{tabular}{c|cccc}
\hline
$\varepsilon$ &$\delta^{\ast}=5h$&$\delta^{\ast}=10h$&$\delta^{\ast}=15h$&$\delta^{\ast}=20h$\\[0.5ex]
\hline
$10^{-1}$&2.888e+1&2.841e+1&2.793e+1&2.746e+1\\
$10^{-2}$&4.761e-1&3.590e-1&2.701e-1&2.036e-1\\
$10^{-3}$&1.340e-2&1.559e-3&1.749e-4&2.045e-5\\
$10^{-4}$&4.264e-5&1.902e-8&1.066e-11&5.108e-15\\
$10^{-5}$&1.789e-4&2.422e-7&3.277e-10&4.435e-13\\
$10^{-6}$&1.834e-4&2.533e-7&3.497e-10&4.827e-13\\
\end{tabular}
\end{table}
\renewcommand{\arraystretch}{1}

\begin{figure}\label{fig:5-3}
\centering
\includegraphics[scale=0.3]{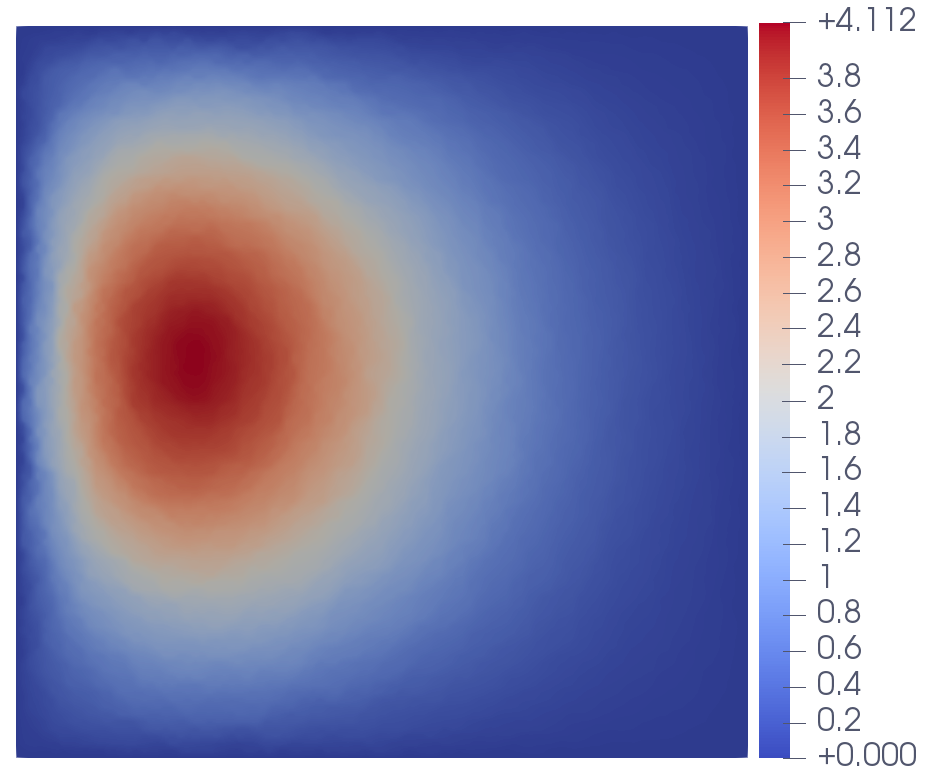}~
\includegraphics[scale=0.3]{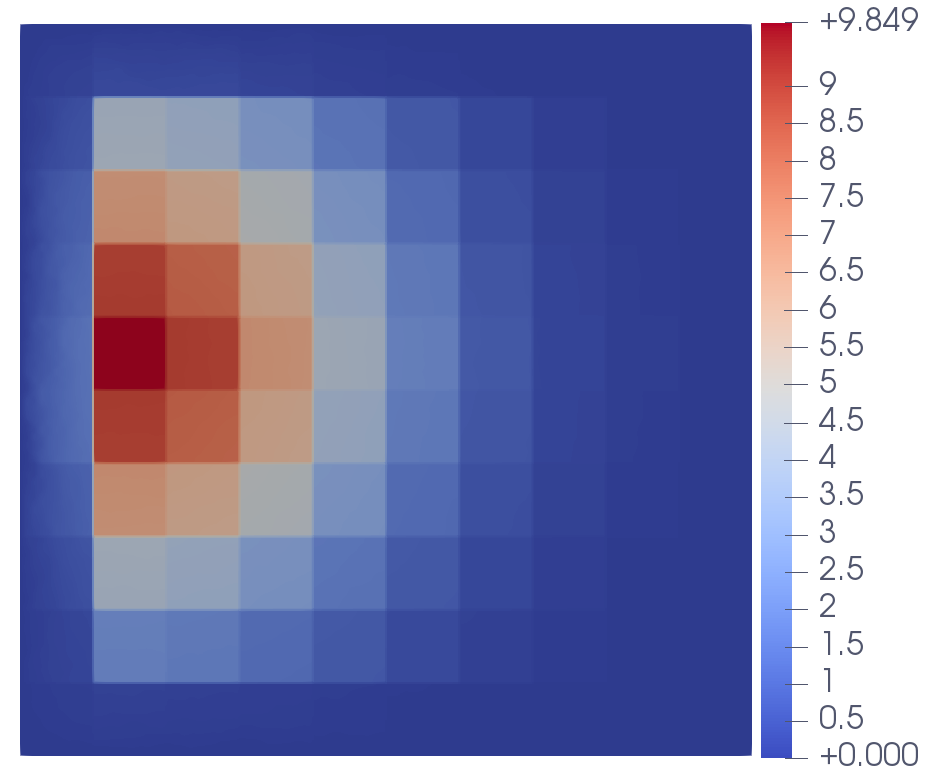}\vspace{1mm}
\includegraphics[scale=0.3]{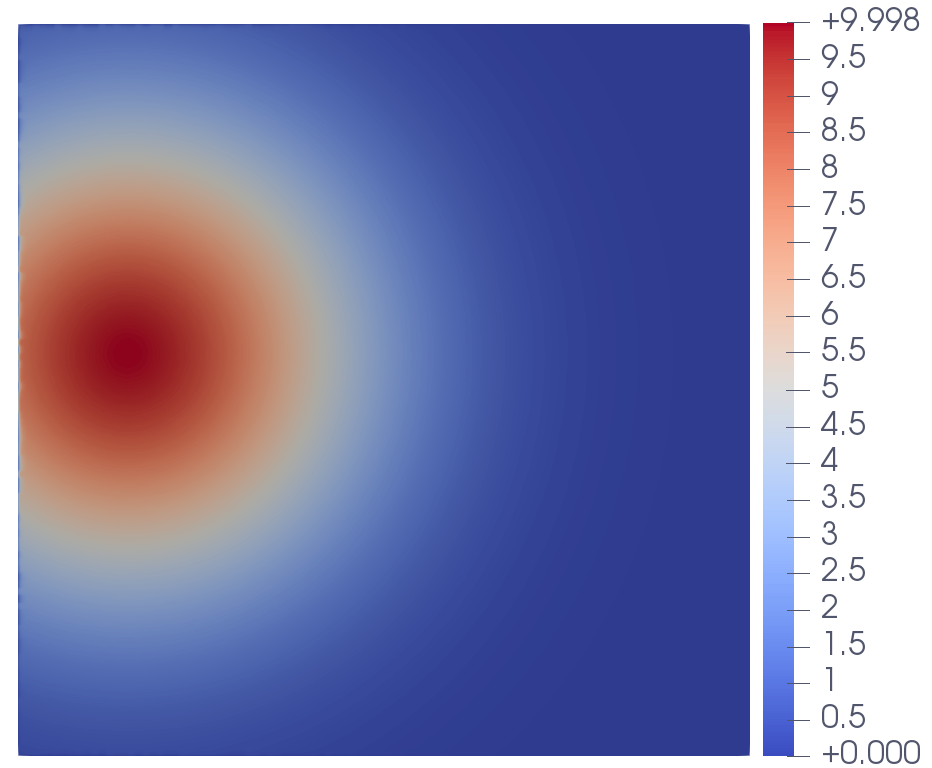}~
\includegraphics[scale=0.3]{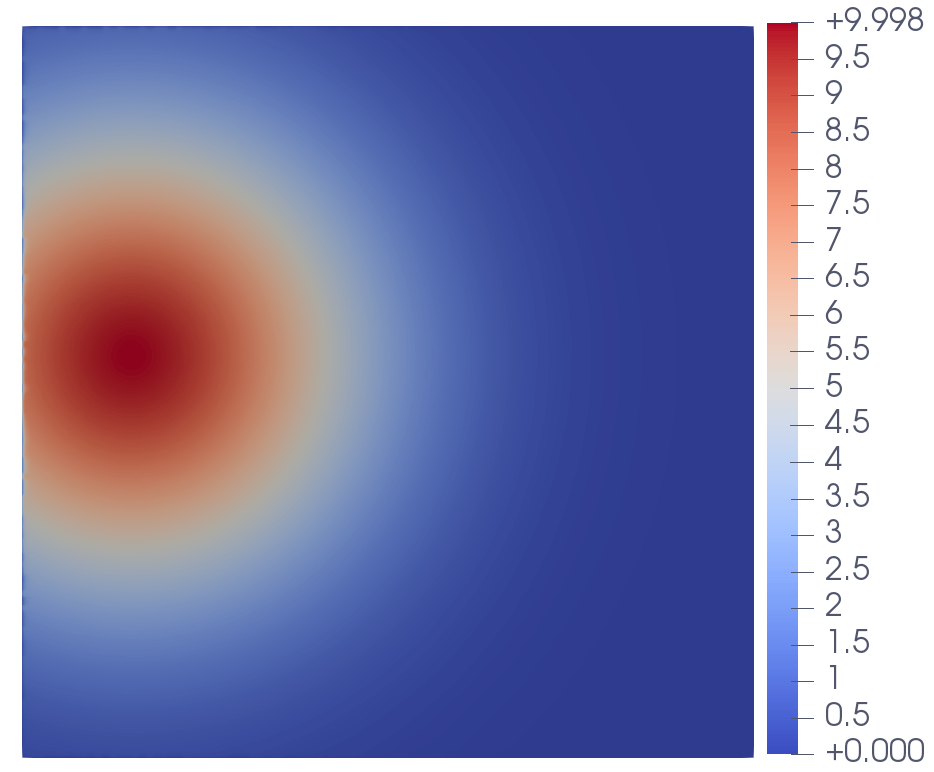}
\caption{The reference solution (left) and the global particular solution (right) for the problem with $\varepsilon=10^{-1}$ (top) and $\varepsilon=10^{-3}$ (down).}
\end{figure}

\section{Conclusions}
We have presented the multiscale spectral GFEM for solving singularly-perturbed reaction-diffusion problems with rough diffusion coefficients. The method was used at the continuous level as a multiscale dicretization scheme of the continous problem, and also at the discrete level as a coarse-space approximation of the fine-scale FE problem. At both levels, $\varepsilon$-explicit and exponential decay rates of the method with respect to the oversampling size and the number of local degrees of freedom were derived. In particular, it was rigorously proved and numerically verified that in the singularly perturbed regime, an accurate approximate solution can be obtained by simply pasting the solutions of local reaction-diffusion problems posed on the oversampling domains together, without using the local spectral basis functions. To the best of our knowledge, this work \mm{is} the first study on multiscale methods for singularly-perturbed heterogeneous reaction-diffusion problems. In the near future, we will investigate the extension of the method to singularly perturbed convection-diffusion problems. Recently, numerical methods for spectral fractional problems were successfully based on numerical methods for singular perturbation problems, \cite{banjai2020exponential}, so that the present work can be the foundation for numerical multiscale fractional diffusion problems.

\appendix
\section{Proof of \cref{lem:3-1}}
\label{sec:A.1}
\begin{proof}
For any $u,v\in H^{1}(\omega_{i}^{\ast})$, the product rule gives 
\begin{equation}\label{eq:3-7}
\begin{array}{lll}
{\displaystyle {a}_{\varepsilon,\omega_{i}^{\ast}}(\eta u,\eta v) = \varepsilon^{2}\int_{\omega_{i}^{\ast}}(A\nabla \eta \cdot \nabla \eta) uv\,d{\bm x} - \varepsilon^{2}\int_{\omega_{i}^{\ast}}(A\nabla u \cdot \nabla \eta) \eta v\,d{\bm x}}\\[4mm]
{\displaystyle \qquad \qquad +\,\varepsilon^{2}\int_{\omega_{i}^{\ast}}(A\nabla \eta \cdot \nabla v) \eta u\,d{\bm x} + {a}_{\varepsilon,\omega_{i}^{\ast}}(u,\eta^{2} v).}
\end{array}
\end{equation}
Exchanging $u$ and $v$ in \cref{eq:3-7} yields that
\begin{equation}\label{eq:3-7-0}
\begin{array}{lll}
{\displaystyle {a}_{\varepsilon,\omega_{i}^{\ast}}(\eta v,\eta u) = \varepsilon^{2}\int_{\omega_{i}^{\ast}}(A\nabla \eta \cdot \nabla \eta) uv\,d{\bm x} - \varepsilon^{2}\int_{\omega_{i}^{\ast}}(A\nabla v \cdot \nabla \eta) \eta u\,d{\bm x}}\\[4mm]
{\displaystyle \qquad \qquad +\,\varepsilon^{2}\int_{\omega_{i}^{\ast}}(A\nabla \eta \cdot \nabla u) \eta v\,d{\bm x} + {a}_{\varepsilon,\omega_{i}^{\ast}}(v,\eta^{2}u).}
\end{array}
\end{equation}
Adding \cref{eq:3-7,eq:3-7-0} together and using the symmetry of the coefficient $A$, we get
\begin{equation}\label{eq:3-8}
{a}_{\varepsilon,\omega_{i}^{\ast}}(\eta u,\eta v) = \varepsilon^{2}\int_{\omega_{i}^{\ast}}(A\nabla \eta \cdot \nabla \eta) uv\,d{\bm x} +\frac{1}{2} \Big( {a}_{\varepsilon,\omega_{i}^{\ast}}(u,\eta^{2} v)+{a}_{\varepsilon,\omega_{i}^{\ast}}(v,\eta^{2}u)\Big).
\end{equation}
By the assumptions on $u$, $v$, and $\eta$, we see that $\eta^{2}u,\eta^{2}v\in H_{0}^{1}(\omega_{i}^{\ast})$. Since $u,v\in H_{a,\varepsilon}(\omega^{\ast}_{i})$, it follows that 
\begin{equation}\label{eq:3-9}
{a}_{\varepsilon,\omega_{i}^{\ast}}(u,\eta^{2} v) = {a}_{\varepsilon,\omega_{i}^{\ast}}(v,\eta^{2}u)=0.
\end{equation}
Inserting \cref{eq:3-9}  into \cref{eq:3-8} gives \cref{eq:3-5}. The inequality \cref{eq:3-6} follows by taking $u=v$ in \cref{eq:3-5} and using \cref{eq:2-2}.
\end{proof}

\bibliographystyle{siamplain}
\bibliography{references}

\begin{thebibliography}{10}

\bibitem{abdulle2014discontinuous}
{\sc A.~Abdulle and M.~E. Huber}, {\em Discontinuous galerkin finite element
  heterogeneous multiscale method for advection--diffusion problems with
  multiple scales}, Numerische Mathematik, 126 (2014), pp.~589--633,
  \url{https://doi.org/10.1007/s00211-013-0578-9}.

\bibitem{babuvska2014machine}
{\sc I.~Babu{\v{s}}ka, X.~Huang, and R.~Lipton}, {\em Machine computation using
  the exponentially convergent multiscale spectral generalized finite element
  method}, ESAIM: Mathematical Modelling and Numerical Analysis, 48 (2014),
  pp.~493--515, \url{https://doi.org/10.1051/m2an/2013117}.

\bibitem{babuska2011optimal}
{\sc I.~Babuska and R.~Lipton}, {\em Optimal local approximation spaces for
  generalized finite element methods with application to multiscale problems},
  Multiscale Modeling \& Simulation, 9 (2011), pp.~373--406,
  \url{https://doi.org/10.1137/100791051}.

\bibitem{babuvska1997partition}
{\sc I.~Babu{\v{s}}ka and J.~M. Melenk}, {\em The partition of unity method},
  International Journal for Numerical Methods in Engineering, 40 (1997),
  pp.~727--758,
  \url{https://doi.org/10.1002/(SICI)1097-0207(19970228)40:4<727::AID-NME86>3.0.CO;2-N}.

\bibitem{bakhvalov1969optimization}
{\sc N.~S. Bakhvalov}, {\em On the optimization of the methods for solving
  boundary value problems in the presence of a boundary layer}, Zhurnal
  Vychislitel'noi Matematiki i Matematicheskoi Fiziki, 9 (1969), pp.~841--859.

\bibitem{banjai2020exponential}
{\sc L.~Banjai, J.~M. Melenk, and C.~Schwab}, {\em Exponential convergence of
  $hp$ fem for spectral fractional diffusion in polygons}, arXiv preprint
  arXiv:2011.05701,  (2020).

\bibitem{solomyak1980quantitative}
{\sc M.~{\v{S}}. Birman and M.~Solomjak}, {\em Quantitative analysis in sobolev
  imbedding theorems and applications to spectral theory}, AMS Translations,
  114 (1980).

\bibitem{brenner2008mathematical}
{\sc S.~C. Brenner, L.~R. Scott, and L.~R. Scott}, {\em The mathematical theory
  of finite element methods}, vol.~3, Springer, 2008.

\bibitem{calo2016multiscale}
{\sc V.~M. Calo, E.~T. Chung, Y.~Efendiev, and W.~T. Leung}, {\em Multiscale
  stabilization for convection-dominated diffusion in heterogeneous media},
  Computer Methods in Applied Mechanics and Engineering, 304 (2016),
  pp.~359--377, \url{https://doi.org/10.1016/j.cma.2016.02.014}.

\bibitem{chung2020multiscale}
{\sc E.~T. Chung, Y.~Efendiev, and W.~T. Leung}, {\em Multiscale stabilization
  for convection--diffusion equations with heterogeneous velocity and diffusion
  coefficients}, Computers \& Mathematics with Applications, 79 (2020),
  pp.~2336--2349, \url{https://doi.org/10.1016/j.camwa.2019.11.002}.

\bibitem{clement1975approximation}
{\sc P.~Cl{\'e}ment}, {\em Approximation by finite element functions using
  local regularization}, Revue fran{\c{c}}aise d'automatique, informatique,
  recherche op{\'e}rationnelle. Analyse num{\'e}rique, 9 (1975), pp.~77--84.

\bibitem{demlow2011local}
{\sc A.~Demlow, J.~Guzm{\'a}n, and A.~Schatz}, {\em Local energy estimates for
  the finite element method on sharply varying grids}, Mathematics of
  computation, 80 (2011), pp.~1--9,
  \url{https://doi.org/10.1090/S0025-5718-2010-02353-1}.

\bibitem{efendiev2011multiscale}
{\sc Y.~Efendiev, J.~Galvis, and X.-H. Wu}, {\em Multiscale finite element
  methods for high-contrast problems using local spectral basis functions},
  Journal of Computational Physics, 230 (2011), pp.~937--955,
  \url{https://doi.org/10.1016/j.jcp.2010.09.026}.

\bibitem{fernando2012numerical}
{\sc H.~Fernando, C.~Harder, D.~Paredes, and F.~Valentin}, {\em Numerical
  multiscale methods for a reaction-dominated model}, Computer methods in
  applied mechanics and engineering, 201 (2012), pp.~228--244,
  \url{https://doi.org/10.1016/j.cma.2011.09.007}.

\bibitem{franca2005convergence}
{\sc L.~P. Franca, A.~L. Madureira, L.~Tobiska, and F.~Valentin}, {\em
  Convergence analysis of a multiscale finite element method for singularly
  perturbed problems}, Multiscale Modeling \& Simulation, 4 (2005),
  pp.~839--866, \url{https://doi.org/10.1137/040608490}.

\bibitem{franca2005towards}
{\sc L.~P. Franca, A.~L. Madureira, and F.~Valentin}, {\em Towards multiscale
  functions: enriching finite element spaces with local but not bubble-like
  functions}, Computer Methods in Applied Mechanics and Engineering, 194
  (2005), pp.~3006--3021, \url{https://doi.org/10.1016/j.cma.2004.07.029}.

\bibitem{harder2015multiscale}
{\sc C.~Harder, D.~Paredes, and F.~Valentin}, {\em On a multiscale hybrid-mixed
  method for advective-reactive dominated problems with heterogeneous
  coefficients}, Multiscale Modeling \& Simulation, 13 (2015), pp.~491--518,
  \url{https://doi.org/10.1137/130938499}.

\bibitem{kadalbajoo2010brief}
{\sc M.~K. Kadalbajoo and V.~Gupta}, {\em A brief survey on numerical methods
  for solving singularly perturbed problems}, Applied mathematics and
  computation, 217 (2010), pp.~3641--3716,
  \url{https://doi.org/10.1016/j.amc.2010.09.059}.

\bibitem{kim2014multiscale}
{\sc M.-Y. Kim and M.~F. Wheeler}, {\em A multiscale discontinuous galerkin
  method for convection--diffusion--reaction problems}, Computers \&
  Mathematics with Applications, 68 (2014), pp.~2251--2261,
  \url{https://doi.org/10.1016/j.camwa.2014.08.007}.

\bibitem{le2017numerical}
{\sc C.~Le~Bris, F.~Legoll, and F.~Madiot}, {\em A numerical comparison of some
  multiscale finite element approaches for advection-dominated problems in
  heterogeneous media}, ESAIM: Mathematical Modelling and Numerical Analysis,
  51 (2017), pp.~851--888, \url{https://doi.org/10.1051/m2an/2016057}.

\bibitem{lin2012balanced}
{\sc R.~Lin and M.~Stynes}, {\em A balanced finite element method for
  singularly perturbed reaction-diffusion problems}, SIAM Journal on Numerical
  Analysis, 50 (2012), pp.~2729--2743, \url{https://doi.org/10.1137/110837784}.

\bibitem{ma2021wavenumber}
{\sc C.~Ma, C.~Alber, and R.~Scheichl}, {\em Wavenumber explicit convergence of
  a multiscale gfem for heterogeneous helmholtz problems}, arXiv preprint
  arXiv:2112.10544,  (2021).

\bibitem{ma2021error}
{\sc C.~Ma and R.~Scheichl}, {\em Error estimates for fully discrete
  generalized fems with locally optimal spectral approximations}, Mathematics
  of Computation, 91 (2022), pp.~2539--2569,
  \url{https://doi.org/10.1090/mcom/3755}.

\bibitem{ma2022novel}
{\sc C.~Ma, R.~Scheichl, and T.~Dodwell}, {\em Novel design and analysis of
  generalized finite element methods based on locally optimal spectral
  approximations}, SIAM Journal on Numerical Analysis, 60 (2022), pp.~244--273,
  \url{https://doi.org/10.1137/21M1406179}.

\bibitem{maalqvist2011multiscale}
{\sc A.~M{\aa}lqvist}, {\em Multiscale methods for elliptic problems},
  Multiscale Modeling \& Simulation, 9 (2011), pp.~1064--1086,
  \url{https://doi.org/10.1137/090775592}.

\bibitem{melenk1997robust}
{\sc J.~M. Melenk}, {\em On the robust exponential convergence of hp finite
  element methods for problems with boundary layers}, IMA Journal of Numerical
  Analysis, 17 (1997), pp.~577--601,
  \url{https://doi.org/10.1093/imanum/17.4.577}.

\bibitem{melenk2002hp}
{\sc J.~M. Melenk}, {\em hp-Finite element methods for singular perturbations},
  no.~1796, Springer Science \& Business Media, 2002.

\bibitem{melenk2016robust}
{\sc J.~M. Melenk and C.~Xenophontos}, {\em Robust exponential convergence of
  hp-fem in balanced norms for singularly perturbed reaction-diffusion
  equations}, Calcolo, 53 (2016), pp.~105--132,
  \url{https://doi.org/10.1007/s10092-015-0139-y}.

\bibitem{netrusov2005weyl}
{\sc Y.~Netrusov and Y.~Safarov}, {\em Weyl asymptotic formula for the
  laplacian on domains with rough boundaries}, Communications in mathematical
  physics, 253 (2005), pp.~481--509,
  \url{https://doi.org/10.1007/s00220-004-1158-8}.

\bibitem{park2004multiscale}
{\sc P.~J. Park and T.~Y. Hou}, {\em Multiscale numerical methods for
  singularly perturbed convection-diffusion equations}, International Journal
  of Computational Methods, 1 (2004), pp.~17--65,
  \url{https://doi.org/10.1142/S0219876204000071}.

\bibitem{pinkus1985n}
{\sc A.~Pinkus}, {\em n-widths in Approximation Theory}, Springer-Verlag,
  Berlin, 1985.

\bibitem{roos2015convergence}
{\sc H.-G. Roos and M.~Schopf}, {\em Convergence and stability in balanced
  norms of finite element methods on shishkin meshes for reaction-diffusion
  problems}, ZAMM-Journal of Applied Mathematics and Mechanics, 95 (2015),
  pp.~551--565, \url{https://doi.org/10.1002/zamm.201300226}.

\bibitem{roos2008robust}
{\sc H.-G. Roos, M.~Stynes, and L.~Tobiska}, {\em Robust numerical methods for
  singularly perturbed differential equations: convection-diffusion-reaction
  and flow problems}, vol.~24, Springer Science \& Business Media, 2008.

\bibitem{schleuss2022optimal}
{\sc J.~Schleu{\ss} and K.~Smetana}, {\em Optimal local approximation spaces
  for parabolic problems}, Multiscale Modeling \& Simulation, 20 (2022),
  pp.~551--582, \url{https://doi.org/10.1137/20M1384294}.

\bibitem{shishkin1989grid}
{\sc G.~Shishkin}, {\em Grid approximation of singularly perturbed boundary
  value problems with a regular boundary layer}, Sov. J. Numer. Anal. Math.
  Model., 4 (1989), pp.~397--417.

\bibitem{spillane2014abstract}
{\sc N.~Spillane, V.~Dolean, P.~Hauret, F.~Nataf, C.~Pechstein, and
  R.~Scheichl}, {\em Abstract robust coarse spaces for systems of pdes via
  generalized eigenproblems in the overlaps}, Numerische Mathematik, 126
  (2014), pp.~741--770, \url{https://doi.org/10.1007/s00211-013-0576-y}.

\bibitem{stein1970singular}
{\sc E.~M. Stein}, {\em Singular integrals and differentiability properties of
  functions}, vol.~2, Princeton university press, 1970.

\bibitem{zhao2022constraint}
{\sc L.~Zhao and E.~Chung}, {\em Constraint energy minimizing generalized
  multiscale finite element method for convection diffusion equation}, arXiv
  preprint arXiv:2203.16035,  (2022).

\end{thebibliography}
\end{document}